\documentclass[14pt]{amsart}
\usepackage{latexsym, amssymb}
\usepackage[usenames]{color}
\usepackage[dvipsnames, svgnames, x11names]{xcolor}
\usepackage[all]{xy}
\usepackage{graphicx,color,tikz}
\usepackage{hyperref}

\newcommand{\be}{\begin{equation}}
\newcommand{\ee}{\end{equation}}
\newcommand{\beq}{\begin{eqnarray}}
\newcommand{\eeq}{\end{eqnarray}}

\newtheorem{thm}{Theorem} [section]

\newtheorem{lem}[thm]{Lemma}

\theoremstyle{definition}

\theoremstyle{remark}
\newtheorem{rmk}[thm]{Remark}

\numberwithin{equation}{section}

\DeclareMathOperator{\ord}{ord}

\DeclareMathOperator{\Fil}{Fil}

\begin{document}
\title[Infinite symmetric power $L$-functions of the hyper-Kloosterman family]
{Infinite symmetric power $L$-functions of the hyper-Kloosterman family}
\begin{abstract}
The infinite symmetric power $L$-functions play a fundamental role in Wan's groundbreaking work on Dwork's conjecture. Building upon this foundation, Haessig established the $p$-adic estimates for these $L$-functions in the case of the one-dimensional Kloosterman family.
In this paper, we extend Haessig’s results by deriving a uniform lower bound for the $q$-adic Newton polygon of the infinite symmetric power $L$-functions associated with the hyper-Kloosterman family.

\end{abstract}

\author[B.L. Wei]{Bolun Wei}
\address{Institute for Math \& AI, Wuhan, Wuhan University, Wuhan 430072, P.R.China}
\email{bolunwei@whu.edu.cn}
\author[L.P. Yang]{Liping Yang}
\address{School of Mathematical Sciences, Chengdu University of Technology, Chengdu 610059, P.R. China}
\email{yanglp2013@126.com}

\maketitle
\section{Introduction}
The hyper-Kloosterman sums are exponential sums associated to the $n$-variables Laurent polynomial
$$f(t,x)=x_1+\cdots+x_n+\frac{t}{x_1\cdots x_n},$$
where $t$ is a parameter. As $t$ varies, this defines a family of exponential sums.  One important approach to studying this family is through symmetric power 
$L$-functions, which we now describe. 

Let $p$ be a prime number and $\mathbb{F}_q$ be the finite field of $q=p^a$ elements, where $a$ is a non-negative integer and ${\rm char}(\mathbb{F}_q)=p$.
Denote by $\bar{\mathbb{F}}_q$ the algebraic closure of $\mathbb{F}_q$. For each $\bar{t}\in \bar{\mathbb{F}}_q$, let $\deg(\bar{t})=[\mathbb{F}_q(\bar{t}):\mathbb{F}_q]$.
Fix a primitive $p$-th root of unity $\zeta_p$, and let $\psi$ be the non-trivial additive character on $\mathbb F_q$ defined by
$$\psi(\bar a)=\zeta_p^{\text{Tr}_{\mathbb F_q/\mathbb F_p}(\bar a)},\ {\rm for}\ \bar a\in \mathbb F_q.$$
For any positive integers $m$, define the hyper-Kloosterman sum ${\rm Kl}_{n}(\bar{t},m)$ by
$${\rm Kl}_{n}(\bar{t},m):=\sum_{ x\in (\mathbb F_{q_{\bar t}^m}^\ast)^n} \psi\circ \text{Tr}_{\mathbb F_{q_{\bar t}^m}/\mathbb F_q}(x_1+\cdots+x_n+\frac{\bar{t}}{x_1\cdots x_n}),$$
where $ q_{\bar{t}}:=q^{\text{deg}(\bar{t})}$.
According to\cite{SP0,SP1}, the generating $L$-function 
$$L(\bar{t},n,T):=\exp\Big(\sum_{m=1}^{\infty}{\rm Kl}_{n}(\bar{t},m)\frac{T^m}{m}\Big) $$
can be written as
$$L(\bar{t},n,T)^{(-1)^{n+1}}=(1-\pi_0(\bar{t})T)\cdots (1-\pi_n(\bar{t})T). $$
If we order the reciprocal roots $\pi_0(\bar{t}),\ldots, \pi_n(\bar{t})$ so that 
$${\rm ord}_{q^{\deg(\bar{t})}}\pi_j(\bar{t})\le {\rm ord}_{q^{\deg(\bar{t})}}\pi_{j+1}(\bar{t}),$$ then 
for $j=0,\cdots,n$, we have
$${\rm ord}_{q^{\deg(\bar{t})}}\pi_j(\bar{t})=j$$
and $\pi_0(\bar{t})$ is a $p$-adic 1-unit, which means that $|1-\pi_0(\bar{t})|_p<1$.

Let $k$ be a positive integer. Define the $k$-th symmetric power $L$-function of the hyper-Kloosterman family by
$$L(Sym^k {\rm Kl}_{n},T):=\prod_{\bar{t}\in |\mathbb{G}_m/\mathbb{F}_q|}\prod_{i_0+\cdots+i_n=k}\frac{1}{1-\pi_0(\bar{t})^{i_0}\cdots\pi_n(\bar{t})^{i_n}T^{\deg(\bar{t})}},$$
where the first product runs over all closed points of the algebraic torus over $\mathbb{F}_q$, and  $i_j$ is nonnegative integer for all $j=0,\ldots,n$.

Robba \cite{Ro86} first studied this $L$-function for $n=1$ using $p$-adic methods, and showed that it is a polynomial over $\mathbb{Z}$. He further conjectured a formula for its degree, which was later proved by Fu and Wan \cite{FW1}. Moreover, they provided an explicit degree formula for the $L$-function associated with the hyper-Kloosterman family by using $l$-adic methods.
In the same work, they raised the question of whether a uniform lower bound exists for the Newton polygon of
$L(Sym^k {\rm Kl}_{n},T)$, independent of $k$.
Haessig \cite{H17} answered this question for the case $n=1$ by establishing such a uniform lower bound. Recently, Fres\'{a}n-Sabbah-Yu showed that this bound can be improved and that it arises from a related exponential mixed Hodge structure, yielding a `Newton above Hodge' phenomenon.
Under certain assumptions, Haessig and Sperber \cite{HS24} proved that the $L$-function $L(Sym^k {\rm Kl}_{n},T)$ for the hyper-Kloosterman family
is a polynomial and also admits a uniform lower bound. Recently, Qin \cite{Qin24}, \cite{Qin24-2} computed the Hodge number for the symmetric power hyper-Kloosterman family via the irregular Hodge theory.

By taking the limit of the $k$-th symmetric power $L$-function, we define the infinite symmetric power $L$-function as follows. Let $\kappa\in \mathbb{Z}_p$, where $\mathbb{Z}_p$ denotes the ring of $p$-adic integers.
The infinite $\kappa$-symmetric power $L$-function is defined by 
 $$L( Sym^{\kappa,\infty}{\rm Kl}_n,T):=\prod_{\bar{t}\in |\mathbb{G}_m/\mathbb{F}_q|} \prod \frac{1}{1-\pi_0(\bar{t})^{\kappa-(i_{1}+\cdots+i_{n})}\pi_{1}(\bar{t})^{i_{1}}\cdots\pi_{n}(\bar{t})^{i_{n}}T^{\deg(\bar{t})}},$$
 where the second product runs over all non-negative integer tuples $i_{1},\cdot\cdot\cdot,i_{n}$.
 Let $\{k_s\}$ be a sequence of positive integers tending to infinite and satisfying  $\lim_{s\rightarrow\infty} k_{s}=\kappa$ $p$-adically.
 Then the following limit holds:
  $$ \lim_{s\rightarrow \infty}L(Sym^{k_s} {\rm Kl}_{n},T)=L( Sym^{\kappa,\infty}{\rm Kl}_n,T).$$

Our study of the infinite symmetric power 
$L$-function is motivated by Dwork's conjecture \cite{Dw73}.
For $\kappa\in \mathbb{Z}_p$, the unit root $L$-function is defined by
  $$L_{unit}(\kappa, T):=\prod_{\bar{t}\in |\mathbb {G}_m/\mathbb{F}_q|}\frac{1}{1-\pi_0(\bar{t})^{\kappa}T^{\deg(\bar{t})}},$$
where $\bar{t}$ runs over all closed points of $\mathbb {G}_m/\mathbb{F}_q$.
 Dwork \cite{Dw73} conjectured that this unit root $L$-function is meromorphic.
Wan's theorems \cite{Wan99} established its meromorphy. In Wan's proof, a key insight is to express the unit root $L$-function as a product of infinite symmetric power $L$-functions. 
It is well known that the unit root $L$-function does not admit a nice cohomology theory.
However, the infinite symmetric power $L$-function  may have one. For the $1$-dimensional case, Haessig \cite{H17} showed that a $p$-adic cohomology theory does exist for the infinite symmetric power $L$-function. By applying the Frobenius  endomorphism to this cohomology, he derived a uniform lower bound for the corresponding $L$-function. Moreover, he obtained an explicit form of the first pole of the unit root $L$-function. In this paper, we generalize Haessig's work to the hyper-Kloosterman family by providing a cohomological description of the infinite symmetric power $L$-function and proving a uniform lower bound for it. To describe the lower bound, we first introduce the following notations.

 Set $$R(T):=\sum_{N=0}^{\infty}\sum_{i_2+\cdots+i_n=N}T^{\sum_{j=2}^nj\cdot i_j}.$$
Write $$R(T)/(1-T^{n+1})=\sum_{i=0}^{\infty}h_i(n)T^i.$$ 
Define the Hodge polygon as the lower convex hull of the points $(0,0)$ and $$(\sum_{i=0}^Nh_i,\sum_{i=0}^N(1-\frac{1}{p-1})ih_i)$$ for $N=0,1,2,\ldots$, which coincides with the $q$-adic Newton polygon of $$\prod(1-q^{(1-\frac{1}{p-1})i}T)^{h_i(n)}.$$
We then prove a `Newton above Hodge' property for the infinite symmetric power $L$-function.
That is,
\begin{thm}\label{thm011}
Let $p$ be an odd prime. For $\kappa \in \mathbb{Z}_p$, the $q$-adic Newton polygon of $L( Sym^{\kappa,\infty}{\rm Kl}_n,T))$ lies on or above the Hodge polygon.
 \end{thm}
 As an application of Theorem \ref{thm011}, we derive a uniform lower bound for the $k$-th symmetric power $L$-function $L(Sym^k {\rm Kl}_n,T)$ on those segments with slopes $\le k$.

 This paper is organized as follows. In section 2, we review the construction of the $p$-adic relative cohomology associated to hyper-Kloosterman sum.
In section 3, we construct a $p$-adic relative cohomology to express the infinite symmetric power $L$-function.
In section 4, we give the proof of Theorem \ref{thm011}. In section 5, we use Theorem \ref{thm011} to give a uniform lower bound for the $k$-th symmetric power $L$-function.

\section {$p$-adic relative cohomology}

Recall that $p$ is a prime and $q=p^a$ with $a\ge 0$. Let $\mathbb{Q}_q$ be the unramified extension of $\mathbb{Q}_p$ of degree $a$, and $\mathbb{Z}_q$ be its ring of integers. Let $E(T):=\exp(\sum_{j=0}^{\infty}\frac{T^{p^j}}{p^j})$ be the Artin-Hasse series, and let $\gamma\in \bar{\mathbb{Q}}_p$ be a root of
$\sum_{j=0}^{\infty}\frac{T^{p^j}}{p^j}$ such that ${\rm ord}_p \gamma=1/(p-1).$ Fix a $q$-th root of $\gamma$, and denoted by $\mathcal{D}_q$ the ring of integers of $\Omega:=\mathbb{Q}_q(\gamma^{1/q})$.

Let $b,b'$ be real numbers such that $\frac{1}{p-1}< b,b'\le \frac{p}{p-1}$.  For $\rho\in \mathbb{R}$, we define
 $$\mathcal{L}_q(b;\rho):=\{\sum_{r=0}^{\infty} a(r)t^r|a(r)\in\Omega, {\rm ord}_p a(r)\ge \frac{b(n+1)r}{q}+\rho\}$$
 and
$$\mathcal{L}_q(b):=\bigcup_{\rho\in\mathbb{R}}\mathcal{L}_q(b;\rho).$$
 For $\mathbf{v}=(v_1,\ldots,v_n)\in \mathbb{Z}^n$, define $m(\mathbf{v}):=\max \{0,-v_1,\cdots,-v_n \}$ and  $w(\mathbf{v}):=\sum_{i=1}^n v_i+(n+1)m(\mathbf{v})$.
 Define
 $$\mathcal{K}_q(b,b';\rho):=\{\sum_{r\geq 0,\mathbf{v}\in \mathbb{Z}^n} a(r,\mathbf{v})t^{r+qm(\mathbf{v}) }x^{\mathbf{v}}|a(r,\mathbf{v})\in \Omega, {\rm ord}_p a(r,\mathbf{v})\ge \frac{b(n+1)r}{q}+b'w(\mathbf{v})+\rho \} $$
 and
 $$\mathcal{K}_q(b,b'):=\bigcup_{\rho\in \mathbb{R}}\mathcal{K}_q(b,b';\rho).$$
 For simplicity, when $q=1$, we write $\mathcal{L}_q(b)$ and $K_q(b,b')$ as $\mathcal{L}(b)$ and $K(b,b')$, respectively.

 The Dwork's infinite splitting function is defined as $\theta(T):=E(\gamma T)=\sum_{j=0}^{\infty}\theta_jT^j$. It is well known that ${\rm ord}_p(\theta_j)\ge j/(p-1)$.
 Define $$F(t,x):=\theta(x_1)\cdots \theta(x_n) \theta(t/x_1\cdots x_n)$$
 and
 $$K(t,x):=\prod_{j=0}^{\infty}F(t^{p^j},x^{p^j}).$$
 For $i\ge 0$, let $\gamma_i:=\sum_{j=0}^i\gamma^{p^j}/p^j=-\sum_{j=i+1}^{\infty}\gamma^{p^j}/p^j$.
 The latter shows that ${\rm ord}_p(\gamma_i)=\frac{p^{i+1}}{p-1}-(i+1)$.
 Let $H(t,x):=\sum_{i=0}^{\infty}\gamma_if(t^{p^i},x^{p^i})$. Then $K(t,x)=\exp H(t,x)$.
For $i=1,\cdots,n$, we define the following differential operators acting on $\mathcal{K}_q(b/q,b')$ by
\begin{align*}
D_{i,t^q}&:=x_i\frac{\partial}{\partial x_i}+x_i\frac{\partial H(t^q,x)}{\partial x_i}\\
&=\frac{1}{K(t^q,x)}\circ x_l\frac{\partial}{\partial x_l}\circ K(t^q,x).
\end{align*}

As $D_{i}$ commutes with each other, one constructs a Koszul complex as follows:
let $\Omega^{0}(\mathcal{K}_q(b,b'),\partial_{t}):=\mathcal{K}_q(b,b')$ and for each $1\le i\le n$, let
$$\Omega^{i}(\mathcal{K}_q(b,b'),D_{t^q}):=\bigoplus_{1\le j_1<j_2<\cdots<j_i\le n}\mathcal{K}_q(b,b') \frac{dx_{j_1}}{x_{j_1}}\wedge\cdots\wedge
\frac{dx_{j_i}}{x_{j_i}} $$
with boundary map $D_{t^q}:\Omega^{i}\rightarrow \Omega^{i+1}$ defined by
$$D_{t^q}(\zeta \frac{dx_{j_1}}{x_{j_1}}\wedge\cdots\wedge
\frac{dx_{j_i}}{x_{j_i}}):=( \sum_{l=1}^nD_{l,t^q}(\zeta)\frac{dx_l}{x_l})\wedge \frac{dx_{j_1}}{x_{j_1}}\wedge\cdots\wedge
\frac{dx_{j_i}}{x_{j_i}}.$$
Denote by $H^i(\mathcal{K}_q(b,b'),D_{t^q})$ the associated cohomology spaces of the Koszul complex $\Omega^{\bullet}(\mathcal{K}_q(b,b'),D_{t^q})$.
Define
$$\mathcal{W}_q(b,b';\rho):=(\Omega[[t]]+\Omega[[t]]x_1+\cdots+\Omega[[t]]x_1\cdots x_n)\cap \mathcal{K}_q(b,b';\rho)$$
and
$$\mathcal{W}_q(b,b'):=\cup_{\rho\in \mathbb{R}} \mathcal{W}_q(b,b';\rho).$$
 It has been proved in \cite[Theorem 3.4]{HS172} that $H^i=0$ for every $i\neq n$, and
 $$\mathcal{K}_q(b,b)=\mathcal{W}_q(b,b)\oplus \sum_{l=1}^n D_{l,t^q}(\mathcal{K}_q(b,b)).$$

Define $\psi_{x}$ acting on the power series ring involving $t$ and $x$ by
$$\psi_{x}\big(\sum a(r,v)t^rx^{v}\big)=\sum a(r, pv)t^r x^{v}.$$
Define
$$\alpha_1(t):=\frac{1}{K(t^p,x)}\circ\psi_x\circ K(t,x)=\psi_{x}\circ F(t,x).$$
As shown in \cite{HS172}, we have $\alpha_1(t):\mathcal{K}(b,b)\rightarrow \mathcal{K}_p(b,b)$ where $\mathcal{K}(b,b)$ stands for $\mathcal{K}_{1}(b,b)$. Then we define
$$\alpha_a:=\alpha_1(t^{p^{a-1}})\cdots \alpha_1(t^{p})\alpha_1(t).$$
Since $q D_{l,t^q}\circ \alpha_a=\alpha_a\circ D_{l,t}$ for each $1\le l\le n$, we can define a chain map ${\rm Frob}^{\bullet}(\alpha_a): \Omega^{\bullet}(\mathcal{K}(b,b),D_t)\rightarrow \Omega^{\bullet}(\mathcal{K}_q(b,b),D_{t^q})$  by
$$ {\rm Frob}^{i}(\alpha_a):=\sum_{1\le j_1<\cdots<j_i\le n}q^{n-i}\alpha_a \frac{dx_{j_1}}{x_{j_1}}\wedge\cdots\wedge
\frac{dx_{j_i}}{x_{j_i}}.$$
Note that $\Omega^{\bullet}(\mathcal{K}_q(b,b),D_{t^q})$ is acyclic except for $H^n$. Then $\alpha_a$ induces a map  from $H^n(\mathcal{K}(b,b),D_{t})\rightarrow H^n(\mathcal{K}_q(b,b),D_{t^q})$, which is denoted by $\bar{\alpha}_a$. 

For $i=0,1,\ldots,n$, let $e_i:=\gamma^ix_1x_2\cdots x_i$. Haessig and Sperber proved in  \cite[Theorem 3.5]{HS172} that $\{e_i\}_{i=0,1,...,n}$ is a basis for the free $\mathcal{L}_q(b)$\text{-}module $H^n(\mathcal{K}_{q}(b,b),D_{t^q})$ and established $p$-adic estimates for the Frobenius  operator. Furthermore, Adolphson and Sperber provided a more precise description of the Frobenius  operator in \cite[Theorem 1.5]{AS2}.
Combining these, we obtain the following conclusion.
\begin{thm}\label{thm2.1}
For $\frac{p}{p-1}\geq b>\frac{1}{p-1}$, let $\{e_i\}_{i=0,1,...,n}$ be a basis of $H^n(\mathcal{K}(b,b),D_{t})$. If we write
$$\bar{\alpha}_1(t)(e_i)=\sum_{j=0}^{n}A(j,i)e_j,$$
then $A(j,i)\in \mathcal{L}(b/p;(b-\frac{1}{p-1})j+(\frac{1}{p-1}-\frac{b}{p})i)$ and $A(0,0)\equiv 1\mod (\gamma,t)$.
\end{thm}

\section{infinite relative cohomology}
In this section, we construct an infinite symmetric power space and define a $p$-adic relative cohomology on it, which will be used to express the infinite $\kappa$-symmetric power $L$-function of the hyper-Kloosterman sums.
The main objective is to provide an explicit description of this cohomology.

For $\underline{i}=(i_1,\cdots,i_n)\in\mathbb{Z}_{\ge 0}^n$, let $|\underline{i}|:=i_1+\cdots+i_n$.
For $\kappa \in \mathbb{Z}_p\setminus \mathbb{Z}_{\ge 0}$, we define $\kappa^{\underline{i}}:=\kappa (\kappa-1)\cdots (\kappa-|\underline{i}|+1)$
and
$$e_{\kappa}^{\underline{i}}:=(e_1^{i_1}\cdots e_n^{i_n})_{\kappa}=\kappa^{\underline{i}}e_1^{i_1}\cdots e_n^{i_n}.$$
Define the weight function of $e_{\kappa}^{\underline{i}}$ by $w(\underline{i}):=1\cdot i_1+\cdots+n\cdot i_n$.

Let $\frac{1}{p-1}<b\le \frac{p}{p-1}$ and $c=b-\frac{1}{p-1}$.
For any $\rho\in \mathbb{R}$, we define the spaces
$$\mathcal{S}(b,c;\rho)=\big\{\sum_{\substack{r\geq 0 \\ \underline{i}\in\mathbb{Z}^n_{\geq0}}}a(r,\underline{i})t^{r}e_{\kappa}^{\underline{i}}|a(r,\underline{i})\in\Omega,\ {\rm ord}_p a(r,\underline{i})\ge b(n+1)r+cw(\underline{i})+\rho\}$$
and
$$\mathcal{S}(b,c)=\bigcup_{\rho\in \mathbb{R}}\mathcal{S}(b,c;\rho).$$
For $\zeta\in \mathcal{S}(b,c;\rho)$, we let
$${\rm ord}_p \zeta=\inf_{\substack{r\geq 0 \\ \underline{i}\in\mathbb{Z}^n_{\geq0}}} \{{\rm ord}_p a(r,\underline{i})\}$$
and 
for $\xi\in \mathcal{S}(b,c)$, let
$${\rm ord}_p \xi=\inf_{\substack{r\geq 0 \\ \underline{i}\in\mathbb{Z}^n_{\geq0}}} \{{\rm ord}_p a(r,\underline{i})-(b(n+1)r+cw(\underline{i}))\}.$$

Recall that $f(t,x)=x_1+\cdots+x_n+t/(x_1\cdots x_n)$ and $H(t,x):=\sum_{i=0}^{\infty}\gamma_if(t^{p^i},x^{p^i})$. Define $\partial:=t\frac{d}{dt}+t\frac{d}{dt}H(t,x)$. Since $\partial$ commutes with each $D_{l,t}$, it induces a connection map $\partial$ on $H^n(\mathcal{K}(b,b),D_{t})$.
We linearly extend $\partial$ to $\mathcal{S}(b,c)$ as follows: for each monomial $t^{r}e_{\kappa}^{\underline{i}}\in \mathcal{S}(b,c;\rho)$, define
\begin{equation*}
\begin{aligned}
\partial(t^{r}e_{\kappa}^{\underline{i}}):=rt^{r}e_{\kappa}^{\underline{i}}+\kappa^{\underline{i}}(\kappa-|\underline{i}|) t^r\partial(e_0)e^{\underline{i}}+\sum_{j=1}^n\kappa^{\underline{i}}i_j t^{r}e_1^{i_1}\cdots e_{j}^{i_j-1}\cdots e_n^{i_n} \partial(e_{j}).
\end{aligned}
\end{equation*}

Let $G$ be an operator on $H^n(\mathcal{K}(b,b),D_{t})$ such that $G(e_i)=e_{i+1}$ for $i=0,\ldots,n-1$ and
 $G(e_n)=\gamma^{n+1}t$, and $\partial^{(1)}:=t\frac{d}{dt}+G$.
Define the map $\nabla_{G}$ acting on $\mathcal{S}(b,c)$
by
\begin{equation*}
\nabla_{G}(t^re_{\kappa}^{\underline{i}}):=\kappa^{\underline{i}}(\kappa-|\underline{i}|) t^rG(e_0)e^{\underline{i}}+\sum_{j=1}^n\kappa^{\underline{i}}i_jt^r e_1^{i_1}\cdots e_{j}^{i_j-1}\cdots e_n^{i_n} G(e_{j}).
\end{equation*}
Similarly, we define $\partial^{(1)}$ acting on $\mathcal{S}(b,c)$ by $\partial^{(1)}:=t\frac{d}{dt}+\nabla_{G}$.

Define the complex $\Omega^{\bullet}(\mathcal{S}(b,c),\partial)$ by
$$\Omega^{0}(\mathcal{S}(b,c),\partial):=\mathcal{S}(b,c)\ {\rm and}\ \Omega^{1}(\mathcal{S}(b,c),\partial):=\mathcal{S}(b,c)\frac{dt}{t}$$
with boundary map $\partial:\Omega^0\rightarrow \Omega^1$ defined by $\eta\mapsto \partial{\eta}\frac{dt}{t}$.
Denote its cohomology by $H^i(\mathcal{S}(b,c),\partial)$
for $i=0,1$.
Then $H^0(\mathcal{S}(b,c),\partial)=\ker \partial$ and $H^1(\mathcal{S}(b,c),\partial)=\mathcal{S}(b,c)/ \partial\mathcal{S}(b,c)$.
Similarly, we define the complex
 $\Omega^{\bullet}(\mathcal{S}(b,c),\partial^{(1)})$ and its cohomology $H^i(\mathcal{S}(b,c),\partial^{(1)})$
for $i=0,1$.
In the following, we will give an explicit description of $H^0(\mathcal{S}(b,c),\partial)$ and $H^1(\mathcal{S}(b,c), \partial)$ for $\kappa\in \mathbb{Z}_p\setminus \mathbb{Z}_{\ge 0}$ by firstly studying $H^0(\mathcal{S}(b,c), \partial^{(1)})$ and $H^1(\mathcal{S}(b,c), \partial^{(1)})$.  


\subsection{$H^0(\mathcal{S}(b,c), \partial^{(1)})$ and $H^1(\mathcal{S}(b,c), \partial^{(1)})$ }
\
\newline
\indent
\begin{lem}\label{lem1}
Let $\kappa \in \mathbb{Z}_p\setminus \mathbb{Z}_{\ge 0}$. We have $H^0(\mathcal{S}(b,c), \partial^{(1)})=0$.
\end{lem}
\begin{proof}
Let $\zeta=\sum_{\underline{i}}\zeta_{\underline{i}}e_{\kappa}^{\underline{i}}\in\mathcal{S}(b,c)$ be the kernel of $\partial^{(1)}$ with $\underline{i}=(i_1,\ldots,i_n)\in \mathbb Z_{\ge 0}^n$ and $\zeta_{\underline{i}}\in \mathcal{L}(b)$.
That is,
\begin{align*}
\sum_{\underline{i}}t\frac{d\zeta_{\underline{i}}}{dt}e_{\kappa}^{\underline{i}}&+\zeta_{\underline{i}}(e_1^{i_1+1}e_2^{i_2}\cdots e_n^{i_n})_{\kappa}
+\sum_{j=1}^{n-1}i_j\zeta_{\underline{i}}(e_1^{i_1}\cdots e_{j}^{i_j-1}e_{j+1}^{i_{j+1}+1}\cdots e_n^{i_n})_{\kappa}\\
&+(\kappa-|\underline{i}|+1)i_n\gamma^{n+1} t\zeta_{\underline{i}}(e_1^{i_1}\cdots e_{n-1}^{i_{n-1}} e_{n}^{i_n-1})_{\kappa}=0.
\end{align*}
Then the coefficient of $e_{\kappa}^{\underline{i}}$ is
\small{\begin{equation}\label{zeq1}
t\frac{d\zeta_{\underline{i}}}{dt}+\zeta_{(i_1-1,i_2,\ldots,i_n)}+
\sum_{j=1}^{n-1}(i_j+1)\zeta_{(i_1,\ldots,i_j+1,i_{j+1}-1,\ldots,i_n)}+(\kappa-|\underline{i}|+1)(i_n+1)\gamma^{n+1} t\zeta_{(i_1,\ldots,i_{n-1},i_n+1)}=0.
\end{equation}}
We proceed by induction on $m$ to show that $t^{m}|\zeta_{\underline{i}}$ for all $m\in\mathbb{Z}_{\geq 1}$.
\par\textbf{Step 1.} For any $\underline{i}\in \mathbb Z_{\geq 0}^n$, we have $t\mid\zeta_{\underline{i}}$.

To simplify the notation, let
\begin{equation} \label{vector1}
\begin{aligned}
\epsilon_{1}&=(2,-1,0,\cdots,0), \\ \epsilon_{2}&=(1,1,-1,0\ldots,0),\\ &\cdots\\
\epsilon_{n-2}&=(1,0,0,\cdots,1,-1,0),\\
\epsilon_{n-1}&=(1,0,0,\ldots,0,1,-1). \end{aligned}
\end{equation}
be vectors in $\mathbb Z^n$, and replace $i_1$ by $i_1+1$ in (\ref{zeq1}). Then $\mathrm{mod}\ t$ property of (\ref{zeq1}) gives us
\begin{equation}\label{zeq2}
\zeta_{\underline{i}}+
(i_1+2)\zeta_{\underline{i}+\epsilon_1}+\sum_{j=2}^{n-1}(i_j+1)\zeta_{\underline{i}+\epsilon_j}\equiv 0\ \mathrm{mod}\ t\ \mathrm{for}\ \mathrm{any}\ \underline{i}\in\mathbb{Z}_{\geq 0}^{n}.
\end{equation}
By letting $i_{2}=\cdots=i_{n}=0$ in (\ref{zeq2}) we obtain $t|\zeta_{(i_{1},0,\cdots,0)}$ for all $i_{1}\geq 0$.

From (\ref{zeq2}), we see that $t|\zeta_{\underline{i}+\epsilon_{j}}$ for all $1\leq j\leq n-1$ implies $t|\zeta_{\underline{i}}$. For a fixed $j$, replace $\underline{i}$ in (\ref{zeq2}) by $\underline{i}+\epsilon_{j}$. We see that $t|\zeta_{\underline{i}+\epsilon_{k}+\epsilon_{j}}$ for all $1\leq k\leq n-1$ implies $t|\zeta_{\underline{i}+\epsilon_{j}}$. It follows that $t|\zeta_{\underline{i}+\epsilon_{k}+\epsilon_{j}}$ for all $1\leq k,j\leq n-1$ implies $t|\zeta_{\underline{i}}$. Recursively using (\ref{zeq2}) we have that for a fixed $N\geq 1$, $t|\zeta_{\underline{i}+\sum_{\alpha=1}^{N}\epsilon_{j_{\alpha}}}$ for all $1\leq j_{1},\cdots,j_{N}\leq n-1$ implies $t|\zeta_{\underline{i}}$. Let $N_{k}$ be the number of $\epsilon_{k}$ appearing in $\{\epsilon_{j_{\alpha}}\}_{1\leq\alpha\leq N}$. Then $N_{1}+\cdots+N_{n-1}=N$.

Now we let $N=(n-1)i_{n}+(n-2)i_{n-1}+\cdots+2i_{3}+i_{2}$. We claim that $t|\zeta_{\underline{i}+\sum_{\alpha=1}^{N}\epsilon_{j_{\alpha}}}$ for all $1\leq j_{1},\cdots,j_{N}\leq n-1$.
If $N_{n-1}>i_{n}$, then the $n$\text{-}th component of $\underline{i}+\sum_{\alpha=1}^{N}\epsilon_{j_{\alpha}}$ is negative.
Hence $\zeta_{\underline{i}+\sum_{\alpha=1}^{N}\epsilon_{j_{\alpha}}}=0$, $t|\zeta_{\underline{i}+\sum_{\alpha=1}^{N}\epsilon_{j_{\alpha}}}$. Thus $N_{n-1}\leq i_{n}$.

Suppose $N_{n-2}>i_{n-1}+N_{n-1}$. Then the $(n-1)$\text{-}th component of $\underline{i}+\sum_{\alpha=1}^{N}\epsilon_{j_{\alpha}}$ is negative.
Hence $\zeta_{\underline{i}+\sum_{\alpha=1}^{N}\epsilon_{j_{\alpha}}}=0$.
It follows that $N_{n-2}\leq N_{n-1}+i_{n-1}$.

 We do above procedure consecutively for $N_{k}$, $1\leq k\leq n-2$. If $N_{k}>N_{k+1}+i_{k+1}$, then the $(k+1)$\text{-}th component in $\underline{i}+\sum_{\alpha=1}^{N}\epsilon_{j_{\alpha}}$ is negative.
 Hence $N_{k}\leq N_{k+1}+i_{k+1}$.

By the above argument, we obtain that $N_{n-1}\leq i_{n}$, $N_{n-2}\leq N_{n-1}+i_{n-1}\leq i_{n-1}+i_{n}$, $N_{n-3}\leq N_{n-2}+i_{n-2}\leq i_{n-2}+i_{n-1}+i_{n}$,$\cdots$, $N_{1}\leq i_{2}+\cdots+i_{n}$. It follows that
$$N=N_{1}+\cdots+N_{n-1}\leq (i_{2}+\cdots+i_{n})+(i_{3}+\cdots+i_{n})+\cdots+(i_{n-1}+i_{n})+i_{n}=N.$$
Hence all above inequalities are forced to be equalities. Note that $\underline{i}+\sum_{\alpha=1}^{N}\epsilon_{j_{\alpha}}=(i_{1}+2i_{2}+\cdots+ni_{n},0,\cdots,0)$.
We have that $t|\zeta_{\underline{i}+\sum_{\alpha=1}^{N}\epsilon_{j_{\alpha}}}$ . Therefore  the claim is true, which implies $t|\zeta_{\underline{i}}$ for any $\underline{i}\in \mathbb{Z}_{\ge 0}^n$.
\par\textbf{Step 2.} If $t^m\mid \zeta_{\underline{i}}$ for any $\underline{i}\in \mathbb{Z}_{\ge 0}^n$, then $t^{m+1}\mid \zeta_{\underline{i}}$ for any $\underline{i}\in \mathbb{Z}_{\ge 0}^n$.

 To simplify the notation, let
\begin{equation} \label{vector2}
\begin{aligned}
\delta_1&=(-1,0,\ldots,0),\\ \delta_2&=(1,-1,0,\ldots,0),\\ &\cdots \\ \delta_n&=(0,\ldots,0,1,-1), \\
\delta_{n+1}&=(0,\ldots,0,0,1).
\end{aligned}
\end{equation}
be vectors in $\mathbb{Z}^{n}$. By the assumption that $t^m\mid \zeta_{\underline{i}}$ for any $\underline{i}\in \mathbb{Z}_{\ge 0}^n$, the $\mathrm{mod}\ t^{m+1}$ property of (\ref{zeq1}) can be written as
\begin{equation}\label{zeq3}
t\frac{d\zeta_{\underline{i}}}{dt}+\zeta_{\underline{i}+\delta_{1}}+
\sum_{j=1}^{n-1}(i_j+1)\zeta_{\underline{i}+\delta
_{j+1}}\equiv 0\ \mathrm{mod}\ t^{m+1}\ \mathrm{for}\ \mathrm{any}\ \underline{i}\in\mathbb{Z}_{\geq 0}^{n}.
\end{equation}
Let $i_{1}=\cdots=i_{n}=0$ in (\ref{zeq3}), we obtain $t^{m+1}|\zeta_{(0,\cdots,0)}.$

Similarly to Step 1, for a fixed integer $N\geq 1$, we have that $t^{m+1}|\zeta_{\underline{i}+\sum_{\alpha=1}^{N}\delta_{j_{\alpha}}}$ for all $1\leq j_{1},\cdots,j_{N}\leq n$ implies $t^{m+1}|\zeta_{\underline{i}}$. Denote $N_{k}$ be the number of $\delta_{k}$ that appears in $\{\delta_{j_{\alpha}}\}_{1\leq\alpha\leq N}$. Then $N_{1}+\cdots+N_{n}=N$.

Now let $N=ni_{n}+(n-1)i_{n-1}+\cdots+2i_{2}+i_{1}$. We claim that $t^{m+1}|\zeta_{\underline{i}+\sum_{\alpha=1}^{N}\epsilon_{j_{\alpha}}}$ for all $1\leq j_{1},\cdots,j_{N}\leq n$.

 Similarly to Step 1, we conclude that $N_{n}\leq i_{n}$, $N_{n-1}\leq N_{n}+i_{n-1}\leq i_{n-1}+i_{n}$, $N_{n-2}\leq N_{n-1}+i_{n-2}\leq i_{n-2}+i_{n-1}+i_{n}$,$\cdots$, $N_{1}\leq i_{1}+\cdots+i_{n}$.Hence
$$N=N_{1}+\cdots+N_{n}\leq (i_{1}+\cdots+i_{n})+(i_{2}+\cdots+i_{n})+\cdots+(i_{n-1}+i_{n})+i_{n}=N.$$
Then all above inequalities are forced to be equalities. Since $\underline{i}+\sum_{\alpha=1}^{N}\delta_{j_{\alpha}}=(0,0,\cdots,0)$, we have that $t^{m+1}|\zeta_{\underline{i}+\sum_{\alpha=1}^{N}\delta_{j_{\alpha}}}$. Thus the claim holds, which implies $t^{m+1}|\zeta_{\underline{i}}$ for all $\underline{i}\in \mathbb{Z}_{\ge 0}^n$.

 Combining with Step 1 and Step 2, we finish the induction on $m$ and  for any $\underline{i}\in \mathbb Z^n_{\geq 0}$
we obtain that $t^m\mid \zeta_{\underline{i}}$ for any $m\in \mathbb Z_{\geq 1}$.
Which means $\zeta_{\underline{i}}=0$ for any $\underline{i}\in \mathbb Z^n_{\geq 0}$. This finishes the proof of Lemma \ref{lem1}.
\end{proof}

For $\rho\in \mathbb{R}$, let $\mathcal{R}(b,c;\rho)$ be a subspace of $\mathcal{S}(b,c;\rho)$ defined by
$$\mathcal{R}(b,c;\rho)=\big\{\sum_{\substack{r\geq 0 \\ \underline{i}\in\mathbb{Z}^n_{\geq 0}\ \mathrm{with}\ i_{1}=0}}a(r,\underline{i})t^{r}e_{\kappa}^{\underline{i}}|a(r,\underline{i})\in\Omega,\ {\rm ord}_p a(r,\underline{i})\ge b(n+1)r+cw(\underline{i})+\rho\}$$
and
$$\mathcal{R}(b,c):=\bigcup_{\rho\in \mathbb{R}}\mathcal{R}(b,c;\rho).$$
The following two lemmas gives the explicit construction of $H^{1}(\mathcal{S}(b,c),\partial^{(1)})$.
\begin{lem} \label{lem5}
For $\kappa \in \mathbb{Z}_p\setminus \mathbb{Z}_{\ge 0}$, we have that $\partial^{(1)}\mathcal{S}(b,c) \cap \mathcal{R}(b,c)=\{0\}.$
\end{lem}
\begin{proof}
Let $\eta\in \partial^{(1)}\mathcal{S}(b,c) \cap \mathcal{R}(b,c)$. Then we write
 $$\eta=\sum_{\substack{r\ge 0\\ \underline{i}=(0,i_2,\cdots,i_n)}} a(r,\underline{i})t^{r}e_{\kappa}^{\underline{i}}.$$
 Let $\zeta=\sum_{\underline{i}\in \mathbb{Z}_{\ge 0}^n}\zeta_{\underline{i}}e^{\underline{i}}_{\kappa}\in \mathcal{S}(b,c;\rho)$ such that
 $\partial^{(1)}(\zeta)=\eta.$

By \hyperref[lem1]{Lemma 3.1} the coefficient of $e^{\underline{i}}_{\kappa}$ in $\partial^{(1)}(\zeta)$ is
\begin{equation*}
\begin{aligned}
t\frac{d\zeta_{\underline{i}}}{dt}+\zeta_{(i_1-1,i_2,\ldots,i_n)}+
\sum_{j=1}^{n-1}(i_j+1)\zeta_{(i_1,\ldots,i_j+1,i_{j+1}-1,\ldots,i_n)}+(\kappa-|\underline{i}|)(i_n+1)\gamma^{n+1} t\zeta_{(i_1,\ldots,i_{n-1},i_n+1)}.
\end{aligned}
\end{equation*}
Comparing the coefficients of $e^{\underline{i}}_{\kappa}$ in $\partial^{(1)}(\zeta)$ and $\eta$, we have that
\begin{align} \label{eq4}
 t\frac{d\zeta_{\underline{i}}}{dt} &+\sum_{j=1}^{n-1}(i_{j}+1)\zeta_{(i_1,\ldots,i_j+1,i_{j+1}-1,\ldots,i_n)}+(\kappa-|\underline{i}|)(i_n+1)\gamma^{n+1}t\zeta_{(i_1,\cdots,i_{n-1},i_n+1)}\\\nonumber
&=\sum_{r\ge 0} a(r,\underline{i})t^r\ \mathrm{when}\ i_{1}=0,
\end{align}
and
\begin{align}\label{eq5}
 t \frac{d\zeta_{\underline{i}}}{dt}&+\zeta_{(i_1-1,i_2,\cdots,i_n)}+\sum_{j=1}^{n-1}(i_j+1)\zeta_{(i_1,\ldots,i_j+1,i_{j+1}-1,\ldots,i_n)} \\\nonumber
 &+
(\kappa-|\underline{i}|)(i_n+1)\gamma^{n+1}t\zeta_{(i_1,\cdots,i_{n-1},i_n+1)}=0\ \mathrm{when}\ i_{1}\geq 1.
\end{align}

We then proceed by induction on $m$ to show that $t^{m}|\zeta_{\underline{i}}$ for all $\underline{i}\in\mathbb{Z}_{\geq 0}^{n}$.
\par\textbf{Step 1.} We show that $t|\zeta_{\underline{i}}$ for all $\underline{i}\in\mathbb{Z}_{\geq 0}^{n}$.

Firstly, it follows from (\ref{eq4}) that
$$\sum_{j=1}^{n-1}(i_j+1)\zeta_{(i_1,\ldots,i_j+1,i_{j+1}-1,\ldots,i_n)}\equiv a(0,\underline{i}) \mod t .$$
Let $i_2=\dots=i_n=0$. Then we have that $a(0,\underline{0})\equiv 0\ \mathrm{mod}\ t$.
 Hence $a(0;\underline{0})=0$.
 By (\ref{eq5}), we have that
\begin{equation*}
\zeta_{(i_1-1,i_2,\cdots,i_n)}+\sum_{j=1}^{n-1}(i_j+1)\zeta_{(i_1,\ldots,i_j+1,i_{j+1}-1,\ldots,i_n)}\equiv 0 \mod t.
\end{equation*}
It then follows that
$\zeta_{(i_1-1,0,\cdots,0)}\equiv 0 \mod t  $
for all $i_1\ge 1$.

Use (\ref{vector1}) to simplify the notation and replace $i_{1}$ by $i_{1}+1$ in (\ref{eq5}).
Then for all $\underline{i}\in\mathbb{Z}_{\geq 0}^{n}$, from (\ref{eq5}) we have that
$$\zeta_{\underline{i}}+
(i_1+2)\zeta_{\underline{i}+\epsilon_1}+\sum_{j=2}^{n-1}(i_j+1)\zeta_{\underline{i}+\epsilon_j}\equiv 0 \mod t.$$
By the same argument as Step 1 of \hyperref[lem1]{Lemma 3.1}, we have $t|\zeta_{\underline{i}}$ for all $\underline{i}\in\mathbb{Z}_{\geq 0}^{n}$.

\textbf{Step 2.} Suppose for each $1\leq k\leq m$, $t^k\mid \zeta_{\underline{i}}$ for all $\underline{i}\in \mathbb{Z}_{\ge 0}^n$. We show that
$t^{m+1}\mid \zeta_{\underline{i}}$ for all $\underline{i}\in \mathbb{Z}_{\ge 0}^n$.

Let $i_{1}=\cdots=i_{n}=0$ in (\ref{eq4}), we have that
$$ t \frac{d\zeta_{\underline{0}}}{dt}+
(\kappa-2)\gamma^{n+1}\zeta_{(0,\cdots,0,1)} t=\sum_{r\ge m}a(r,(0,\cdots,0))t^r.$$
By the assumption $t^m\mid \zeta_{\underline{i}}$ for all $\underline{i}\in \mathbb{Z}_{\ge 0}^n$, the coefficient of $t^{m}$ in $\zeta_{\underline{0}}$ is $a(m;\underline{0})/m$.
Let $i_1\ge 1$ and $i_2=\cdots=i_n=0$ in (\ref{eq5}). We obtain that
$$t \frac{d\zeta_{(i_1,0,\cdots,0)}}{dt}+\zeta_{(i_1-1,0,\cdots,0)}\equiv 0 \mod t^{m+1}. $$
Recursively using the above relation, we compute that the coefficient of $t^m$ in $\zeta_{(i_1,0,\cdots,0)}$
is $(-1)^{i_1}a(m;\underline{0})/m^{i_1+1}$.
Hence $$\zeta_{(i_1,0,\cdots,0)}=(-1)^{i_1}\frac{a(m;\underline{0})}{m^{i_{1}+1}}t^m+O(t^{m+1})$$
for each $i_1\ge 0$, which means that
$a(m;\underline{0})/m^{i_1+1}\rightarrow 0 $ as $i_1\rightarrow \infty$ due to the Banach algebra structure of $\mathcal{S}(b,c;\rho)$. This is impossible unless $a(m;\underline{0})=0$. Therefore, we have $t^{m+1}|\zeta_{(i_{1},0,\cdots,0)}$ for all $i_{1}\geq 0$.

To prove the argument for any $\underline{i}\in \mathbb Z_{\geq 0}^n$. We define 
$$I(\underline{i}):=\omega(\underline{i})-|\underline{i}|=i_{2}+2i_{3}+\cdots (n-1)i_{n},$$
and proceed with an induction on $I(\underline{i})$ to show $t^{m+1}|\zeta_{\underline{i}}$ for any $I(\underline{i})\in\mathbb{Z}_{\geq 0}$. The above discussion shows that the argument is true for $I(\underline{i})=0$. Suppose $t^{m+1}|\zeta_{\underline{i}}$ for any $\underline{i}\in\mathbb{Z}_{\geq 0}^{n}$ satisfying $I(\underline{i})=l$. We will show that $t^{m+1}|\zeta_{\underline{i}}$ for any $\underline{i}$ such that $I(\underline{i})=l+1$.

 Note that the value of $I(\underline{i})$ is independent of the value of $i_{1}$. Given $(0,i_{2},\cdots,i_{n})\in\mathbb{Z}_{\geq 0}^{n}$ such that $I(0,i_{2},\cdots,i_{n})=l+1$, then $I(\underline{i})=l+1$ with $\underline{i}=(i_{1},i_{2},\cdots,i_{n})$ for any $i_{1}\geq 0$.

For $i_{1}=0$, applying (\ref{eq4}) we have that 
$$t\frac{d\zeta_{\underline{i}}}{dt}+\sum_{j=1}^{n-1}(i_{j}+1)\zeta_{\underline{i}+\delta_{j+1}}+(\kappa-|\underline{i}|)(i_{n}+1)\gamma^{n+1}t\zeta_{\underline{i}+\delta_{n+1}}=\sum_{r\geq m}a(r;\underline{i})t^{r}.$$
Notice that $I(\underline{i}+\delta_{j})=l$ for all $1\leq j\leq n$. Then the induction hypothesis on $m$ and on $I(\underline{i})$ gives 
$$t\frac{d\zeta_{\underline{i}}}{dt}\equiv a(m;\underline{i})t^{m}\ \mathrm{mod}\ t^{m+1}.$$
So the coefficient of $t^{m}$ in $\zeta_{(0,i_{2},\cdots,i_{n})}$ is $a(m;(0,i_{2},\cdots,i_{n}))/m$.

For $i_{1}\geq 1$, applying (\ref{eq5}) and the induction  we obtain
$$t\frac{d\zeta_{\underline{i}}}{dt}+\zeta_{(i_{1}-1,i_{2},\cdots,i_{n})}\equiv 0\ \mathrm{mod}\ t^{m+1}.$$
The recursive argument gives
$$\zeta_{(i_{1},i_{2},\cdots,i_{n})}=(-1)^{i_{1}}\frac{a(m;(0,i_{2},\cdots,i_{n}))}{m^{i_{1}+1}}+O(t^{m+1}).$$
Again, the Banach structure gives $a(m;(0,i_{2},\cdots,i_{n}))=0$. Hence $t^{m+1}|\zeta_{(i_{1},i_{2},\cdots,i_{n})}$ for any $i_{1}\geq 0$. Moreover, for any such $\underline{i}$ with $I(\underline{i})=l+1$. This finishes the induction on $I(\underline{i})$, and completes Step 2. 

In summary, we obtain $t^{m}|\zeta_{\underline{i}}$ for any $m\geq 1$ and any $\underline{i}\in\mathbb{Z}_{\geq 0}^{n}$. Hence $\zeta=0$ and $\eta=\partial^{(1)}(\zeta)=0$, which finishes the proof of Lemma \ref{lem5}.
\end{proof}

\begin{lem} \label{lem2}
For $\kappa \in \mathbb{Z}_p\setminus \mathbb{Z}_{\ge 0}$, we have that
$$\mathcal{S}(b,c;0)=\mathcal{R}(b,c;0)+  \nabla_{G}\mathcal{S}(b,c;c).$$
\end{lem}
\begin{proof}
It is enough to show that $\mathcal{S}(b,c;0)\subseteq\mathcal{R}(b,c;0)+  \nabla_{G}\mathcal{S}(b,c;c).$
By the definition of $\nabla_{G}$, we have that
\begin{align}\label{eq8}
e_{\kappa}^{\underline{i}}=&\nabla_{G}(e_1^{i_1-1}e_2^{i_2}\cdots e_n^{i_n})_{\kappa}-\big((i_1-1)(e_1^{i_1-2}e_2^{i_2+1}\cdots e_n^{i_n})_{\kappa}+\sum_{j=2}^{n-1}i_j(e_1^{i_1-1}e_2^{i_2}\cdots e_j^{i_j-1}e_{j+1}^{i_{j+1}+1}\cdots e_n^{i_n})_{\kappa}\\\nonumber
&+(\kappa-|\underline{i}|+2)i_n\gamma^{n+1} t(e_1^{i_1-1}e_2^{i_2}\cdots e_{n-1}^{i_{n-1}}e_n^{i_n-1})_{\kappa}\big).
\end{align}

We use (\ref{vector1}) to simplify the notation, and let $\epsilon_0=(1,0,...,0)$ and $\epsilon_n=(1,0,...,1)$.
Then (\ref{eq8}) can be rewritten as
\begin{align*}
e_{\kappa}^{\underline{i}}=\nabla_{G}(e^{\underline{i}-\epsilon_0}_{\kappa})+\big((1-i_1)e_{\kappa}^{\underline{i}-\epsilon_1}+\sum_{j=2}^{n-1}(-i_j)e_{\kappa}^
{\underline{i}-\epsilon_j}+(|\underline{i}|-2-\kappa)i_n\gamma^{n+1} te_{\kappa}^{\underline{i}-\epsilon_n}\big).
\end{align*}
Then using (\ref{eq8}) recursively, we have that
\begin{align*}
e_{\kappa}^{\underline{i}}=&\nabla_{G}(\sum_{l_1,\ldots,l_n\ge 0}\zeta_{l_1,\ldots,l_n}\gamma^{(n+1)l_n}t^{l_n}e^{\underline{i}-\sum l_i\epsilon_i-\epsilon_0}_{\kappa})\\
&+\sum_{\substack{l_1,\ldots,l_n\ge 0\\ i_1=2l_1+l_2+\ldots+l_n}}\delta_{l_1,\ldots,l_n}\gamma^{(n+1)l_n}t^{l_n}
e^{\underline{i}-\sum l_i\epsilon_i}_{\kappa},
\end{align*}
where $\zeta_{l_1,\ldots,l_n}$ and $\delta_{l_1,\ldots,l_n}$ are $p$-adic integers.

Let $\zeta=\sum_{r,\underline{i}}a(r,\underline{i})t^re_{\kappa}^{\underline{i}} \in \mathcal{S}(b,c;0)$. Then we have the following
\begin{align*}
\zeta=&\sum_{r\geq 0,\ \underline{i}\in\mathbb{Z}_{\geq 0}^{n}}a(r,\underline{i})t^r\Big( \nabla_{G}(\sum_{l_1,\ldots,l_n\ge 0}\zeta_{l_1,\ldots,l_n}\gamma^{(n+1)l_n}t^{l_n}e^{\underline{i}-\sum l_i\epsilon_i-\epsilon_0}_{\kappa})\\
&+\sum_{\substack{l_1,\ldots,l_n\ge 0\\ i_1=2l_1+l_2+\ldots+l_n}}\delta_{l_1,\ldots,l_n}\gamma^{(n+1)l_n}t^{l_n}
e^{\underline{i}-\sum l_i\epsilon_i}_{\kappa}\Big)\\
=&\nabla_{G}\Big(\sum_{s\geq 0,\ \underline{j}\in\mathbb{Z}_{\geq 0}^{n}}a^{(1)}(s,\underline{j})t^{s}e_{\kappa}^{\underline{j}}\Big)+\sum_{\substack{s'\geq 0,\ \underline{j}'\in\mathbb{Z}_{\geq 0}^{n} \\{\rm such\ that}\  j'_1=0}}a^{(2)}(s',\underline{j}')t^{s'}e_{\kappa}^{\underline{j}'}
\end{align*}
where 
\begin{equation*}
\begin{aligned}
a^{(1)}(s,\underline{j})&=\sum_{\substack{l_1,\ldots,l_{n}\ge 0,\ l_{n}\leq s\\ \mathrm{such}\ \mathrm{that}\ \underline{j}+\sum l_i\epsilon_i+\epsilon_0\in\mathbb{Z}_{\geq 0}^{n}}} a(s-l_{n},\underline{j}+\sum l_i\epsilon_i+\epsilon_0) \gamma^{(n+1)l_n}\zeta_{l_1,\ldots,l_n}  \\
a^{(2)}(s',\underline{j}')&=\sum_{\substack{l'_1,\ldots,l'_n\ge 0,\ l'_{n}\leq s'\\ \mathrm{such}\ \mathrm{that}\ \underline{j}'+\sum l'_i\epsilon_i\in\mathbb{Z}_{\geq 0}^{n}}} a(s'-l'_{n},\underline{j}'+\sum l'_i\epsilon_i) \gamma^{(n+1)l_n}\delta_{l'_1,\ldots,l'_n}.
\end{aligned}
\end{equation*}

Note that the summation restriction in $a^{(1)}(s,\underline{j})$ is equivalent to $0\leq l_{n}\leq s$ and $0\leq l_{k-1}\leq l_{k}+j_{k}$ for $2\leq k\leq n$, which are finite summations. Hence it is well-defined in $\Omega$. Similarly, we have that $a^{(2)}(s',\underline{j}')$ is also well-defined. Recall that $c=b-\frac{1}{p-1}$. One can check that
if $a(r,i)t^re_{\kappa}^{\underline{i}}\in \mathcal{S}(b,c;0)$, then
$$a^{(1)}(s,\underline{j})t^{s}e^{\underline{j}}_{\kappa}\in \mathcal{S}(b,c;c)\ \mathrm{and}\ a^{(2)}(s,\underline{j})t^{s}e^{\underline{j}}_{\kappa}\in \mathcal{R}(b,c;0).$$
This finishes the proof of Lemma \ref{lem2}.
\end{proof}

\begin{thm}\label{thm1}
For $\kappa \in \mathbb{Z}_p\setminus \mathbb{Z}_{\ge 0}$, we have that
$$\mathcal{S}(b,c;0)=\mathcal{R}(b,c;0)\oplus  \partial^{(1)}\mathcal{S}(b,c;c).$$
\end{thm}
\begin{proof}
By Lemma \ref{lem5}, it suffices to prove $$\mathcal{S}(b,c;0)\subseteq \mathcal{R}(b,c;0)+ \partial^{(1)}\mathcal{S}(b,c;c).$$
Let $\zeta\in \mathcal{S}(b,c;0)$. It follows from Lemma \ref{lem2} that there exists $\eta_0\in \mathcal{R}(b,c;0)$
and $\epsilon_0\in \mathcal{S}(b,c;c)$ such that
$$\zeta=\eta_0+\nabla_G(\epsilon_0).$$
Let $\zeta_1=-t\frac{d \epsilon_0}{dt}$. Then
$$\zeta=\eta_0+\partial^{(1)}(\epsilon_0)+\zeta_1.$$
Since $\zeta_1\in \mathcal{S}(b,c;c)$ which increase by $c$ in valuation, there exists $\eta_1\in \mathcal{R}(b,c;c)$
and $\epsilon_1\in \mathcal{S}(b,c;2c)$ such that
$$\zeta_1=\eta_1+\nabla_G\epsilon_1.$$
Let $\zeta_1=-t\frac{d \epsilon_1}{dt}$.
Then
$$\zeta=\eta_0+\eta_1+\partial^{(1)}(\epsilon_0+\epsilon_1).$$
We repeat the process to obtain
$\zeta=\sum_{m\ge 0}\eta_m+\partial^{(1)}\sum_{m\ge 0}\epsilon_m$, with $\eta_m\in \mathcal{R}(b,c;mc)$ and
$\epsilon_m\in \mathcal{S}(b,c;(m+1)c)$. Hence $\sum_{m\ge 0}\eta_m\in \mathcal{R}(b,c;0)$ and $\sum_{m\ge 0}\epsilon_m\in \mathcal{S}(b,c;c)$ are well defined.
It follows that Theorem \ref{thm1} holds.
\end{proof}
\begin{rmk}\label{remark}
It follows from Lemma \ref{lem1} that $H^{0}(\mathcal{S}(b,c),\partial^{(1)})=0$ and $H^{1}(\mathcal{S}(b,c),\partial^{(1)})\simeq\mathcal{R}(b,c)$. From (\ref{eq8}) and Theorem \ref{thm1} we can conclude that if $\zeta\in \mathcal{S}(b,c;0)$ with ${\rm ord}_p\zeta \ge s$, then there exist $\eta\in \mathcal{R}(b,c;0)$ and $ \varepsilon\in \mathcal{S}(b,c;c)$ such that
$\zeta=\eta+\partial^{(1)}(\varepsilon)$, where ${\rm ord}_p\eta\ge s$ and ${\rm ord}_p\varepsilon\ge s$.
\end{rmk}


\subsection{$H^0(\mathcal{S}(b,c), \partial)$ and $H^1(\mathcal{S}(b,c), \partial)$ }
\
\newline
\indent

In this subsection, we will show that $H^0(\mathcal{S}(b,c), \partial)=0$ and $H^1(\mathcal{S}(b,c), \partial)\cong
\mathcal{R}(b,c)$ for $\kappa\in \mathbb{Z}_p\setminus \mathbb{Z}_{\ge 0}$.
Denote by $\mathcal{F}$ the subset of $\prod_{i=1}^{\infty}\mathbb{Z}_{\ge 0}$ where every element $\mathbf{j}\in \mathcal{F}$ is of the form $\mathbf{j}=(j_1,\ldots,j_s,0,0,\ldots)$ for some $s\ge 1$, and for each $j_i>0$. Define $s(\mathbf{j}):=\max\{i:j_i>0\}$. Define $\rho(\mathbf{j}):=p^{j_1}+\cdots+p^{j_s}-s(\mathbf{j})$.

For $q=p^a$ with $a\ge 0$, we introduce the following spaces
$$\mathcal{O}_{q}:=\{\sum_{r\ge 0} a(r)\gamma^{(n+1)r/q}t^{r}:a(r)\in \mathcal{D}_q, a(r)\rightarrow 0 {\rm\ as}\  r\rightarrow \infty\}$$
and
$$\mathcal{C}_{q}:=\{\sum_{\mathbf{u}} \zeta(\mathbf{u})\gamma^{w(\mathbf{u})}t^{qm(\mathbf{u})}x^{\mathbf{u}}:\zeta(\mathbf{u})\in \mathcal{O}_{q}, \zeta(\mathbf{u})\rightarrow 0 {\rm\ as}\  w(\mathbf{u})\rightarrow \infty\}.$$
For $q=1$, we simply write $\mathcal{O}_{q}$ and $\mathcal{C}_{q}$ as  $\mathcal{O}$ and $\mathcal{C}$, respectively.

 Note that for $\frac{1}{p-1}< b,b'\le\frac{p}{p-1}$, we have $\mathcal{K}_{q}(b,b';0)\subseteq \mathcal{C}_{q}$. We can define the complex $\Omega^{\bullet}(\mathcal{C}_{q},D_{t^{q}})$ similarly as 
 $\Omega^{\bullet}(\mathcal{K}_q(b,b),D_{t^{q}})$. As \cite[Theorem 3.2]{HS172} showed that $H^{i}(\mathcal{C}_{q},D_{t^{q}})$ is acyclic except $i=n$, where $H^{n}(\mathcal{C}_{q},D_{t^{q}})$ is a free $\mathcal{O}_{q}\text{-}$module of rank $n+1$ with the same basis as $H^{n}(\mathcal{K}_{q}(b,b),D_{t^{q}})$. 
 
 Define the following ring
$$\mathfrak{M}_{t}:=\{f(t)\in\Omega[[t]]: f\ \mathrm{is}\ p\text{-}{\rm adically}\ \mathrm{analytic}\ \mathrm{in\ the\ neighbourhood\  of\ }\ 0\}.$$
We extend scalars by lifting the base coefficient ring to $\mathfrak{M}_t$. Then we have 
\begin{equation} \label{relation}
H^{n}(\mathcal{C}_{q},D_{t^{q}})\otimes_{\mathcal{O}_{q}}\mathfrak{M}_{t}=H^{n}(\mathcal{K}_{q}(b,b),D_{t^{q}})\otimes_{\mathcal{L}_{q}(b)}\mathfrak{M}_{t}    
 \end{equation}
 as free $\mathfrak{M}_{t}\text{-}$modules. 
 
For monomial $t^{qm(\mathbf{u})+r}x^{\mathbf{u}}$, we let $W_q(r;\mathbf{u})=(n+1)r/q+w(\mathbf{v})$.
We define filtrations on $\mathcal{C}_{q}$ and $\mathcal{K}_{q}(b,b')$ as follows:
 $$\Fil^{N}\mathcal{C}_{q}:=\{\mathcal{D}_q\text{-}{\rm module\ generated\ by\ }  \gamma^{\frac{(n+1)r}{q}+w(\mathbf{u})}t^{qm(\mathbf{u})+r}x^{\mathbf{u}}{\rm such\ that\ } W_{q}(r;\mathbf{u})\le N\}$$
 and 
 $$\Fil^{N}\mathcal{K}_{q}(b,b';\rho):=\{\sum_{\substack{r\geq 0,\mathbf{v}\in \mathbb{Z}^n \\ \mathrm{such}\ \mathrm{that}\ W_q(r;\mathbf{v})\leq N}} a(r,\mathbf{v})t^{r+qm(\mathbf{v}) }x^{\mathbf{v}}\in\mathcal{K}_{q}(b,b';\rho)\}.$$
For $\zeta\in\Fil^{N}\mathcal{K}_{q}(b,b;0)\subset \Fil^{N}\mathcal{C}_{q}$, Haessig and Sperber \cite [Theorem 4.3]{HS24} proved that there exists $c(l,\zeta)\in \mathcal{L}(\frac{1}{q}(1+\frac{1}{p-1});l-N)$ such that
$$\zeta=\sum_{l=0}^n c(l,\zeta)e_l+\sum_{i=1}^nD_{i,t^{q}}(\zeta_i)$$
with $\zeta_i\in \mathcal{C}_{q}$ and 
$$\sum_{l=0}^n c(l,\zeta)e_l=\sum_{l=0}^{\min\{N,n\}} a(l,\zeta)e_l+\sum_{\mathbf{j}\in \mathcal{F}}\sum_{l}p^{\rho(\mathbf{j})}a(l,\zeta;\mathbf{j})e_l.$$
On the other hand, by \cite[Theorem 3.4]{HS172}, we also have that $\zeta$ can be written as 
 $$\zeta=\sum_{l=0}^n b(l,\zeta)e_l+\sum_{l=1}^nD_{l,t^{q}}(\eta_l)$$
 with $b(l,\zeta)\in \mathcal{L}(\frac{b}{q},0)$  and $\eta_l\in \mathcal{K}_{q}(b,b;0)\subset \mathcal{C}_{q}$. Then we have $c(l,\zeta)=b(l,\zeta)$ for $l=0,1,\ldots,n$. We can similarly define $\partial$ and $\partial^{(1)}$ acting on $H^n (\mathcal{C},D_t)$ as that on $H^n (\mathcal{K}(b,b),D_t)$. Then we have:

\begin{lem}\label{lem3}
For each $0\le l\le n$, we may write in $H^n(\mathcal{C},D_{t})$
$$\partial^{(1)}(e_{l})=\partial(e_{l})+\sum_{\mathbf{j}\in \mathcal{F}}\sum_{m\le \min\{\rho(\mathbf{j})+l+1,n\}}p^{\rho(\mathbf{j})}b(m,l;\mathbf{j})e_m \mod D_{t},$$
where $b(m,l;\mathbf{j})=\sum_{r(n+1)\le \rho(\mathbf{j})+l-m+1} a(r)\gamma^{(n+1)r} t^r$ with ${\rm ord}_p a(r)\ge 0$.
\end{lem}
\begin{proof}
Let $0\le l\le n-1$. We have that $\partial^{(1)}(e_{l})=G(e_l)=e_{l+1}$.
From the definition we have that
$$\partial(e_{l})=e_lt\frac{d}{dt}H(t,x)=e_l(\frac{\gamma t}{x_1\cdots x_n}+\sum_{i=1}^{\infty}\gamma_ip^iK_t(t^{p^i},x^{p^i})), $$
where $K_t(t,x)=\frac{t}{x_1\cdots x_n}$.
Note that
\begin{align*}
D_{l+1,1}(e_l)&=e_lx_{l+1}\frac{\partial}{d x_{l+1}}H(t,x)\\
&=e_l(\gamma x_{l+1}-\frac{\gamma t}{x_1\cdots x_n}+\sum_{i=1}^{\infty}\gamma_i p^ix_{l+1}^{p^i}-\sum_{i=1}^{\infty}\gamma_ip^iK_t(t^{p^i},x^{p^i})).
\end{align*}
Hence
$$\partial(e_l)=\partial^{(1)}(e_l)+\sum_{i=1}^{\infty}\gamma_i p^ix_{l+1}^{p^i}e_l \mod D_t$$
for $0\le l\le n-1$.
For $l=n$, we have that
$$\partial(e_{n})=e_n(\frac{\gamma t}{x_1\cdots x_n}+\sum_{i=1}^{\infty}\gamma_ip^iK_t(t^{p^i},x^{p^i}))=\partial^{(1)}(e_n)+\sum_{i=1}^{\infty}\gamma_ip^iK_t(t^{p^i},x^{p^i})e_n.$$
We may write $\gamma_ip^i=p^{p^i-1}\tau_i\gamma^{p^i}$ with ${\rm ord}_p \tau_i=0$.
Let $\nu_{p^i,l}:=-\tau_i \gamma^{p^i}x_{l+1}^{p^i}$ for $0\le l\le n-1$ and $\nu_{p^i,n}:=-\tau_i\gamma^{p^i}K_t(t^{p^i},x^{p^i})$. Then $\nu_{p^i,l}\in\Fil^{p^i}\mathcal{C}$.
By the same argument as \cite[Lemma 4.4(2)]{HS24} we have that Lemma \ref{lem3} holds.


\end{proof}

\begin{thm}\label{thm3}
Let $\kappa \in \mathbb{Z}_p\setminus \mathbb{Z}_{\ge 0}$.
Then $\mathcal{S}(b,c;0)=\mathcal{R}(b,c;0)+ \partial\mathcal{S}(b,c;c).$
\end{thm}
\begin{proof}

 For $\zeta\in \mathcal{S}(b,c;0)$, by Theorem \ref{thm1}, there exist $a(r,\underline{i})\in \Omega$ such that ${\rm ord}_pa(r,\underline{i})\ge b(n+1)r+cw(\underline{i})$ and
$\eta_0\in \mathcal{S}(b,c;c)$ satisfying that
$$\zeta=\sum_{r\ge 0,\underline{i}=(0,i_2,\cdots,i_n)}a(r,\underline{i})t^re_{\kappa}^{\underline{i}}+\partial^{(1)}(\eta_0).$$
Write $\eta_0=\sum_{s\ge 0,\underline{i}\in \mathbb{Z}_{\ge 0}^n}b(s,\underline{i})t^se_{\kappa}^{\underline{i}}$ with ${\rm ord}_pb(s,\underline{i})\ge b(n+1)s+cw(\underline{i})+c$.

From the definition of $\partial^{(1)}$, we have that
\begin{align*}
\partial^{(1)}(t^se_{\kappa}^{\underline{i}})
&=t\frac{d(t^se_{\kappa}^{\underline{i}})}{dt}+t^s\partial^{(1)}(e_{\kappa}^{\underline{i}})\\
&=t\frac{d(t^se_{\kappa}^{\underline{i}})}{dt}+\kappa^{\underline{i}}t^s\big((\kappa-|\underline{i}|)e^{\underline{i}}\partial^{(1)}(e_0)+
\sum_{l=1}^ni_l e_1^{i_1}\cdots e_{l}^{i_l-1}\cdots e_n^{i_n}\partial^{(1)}(e_l)\big).
\end{align*}
From (\ref{relation}), we first lift the basis $\{e_{i}\}_{0\leq i\leq n}$ of $H^{n}(\mathcal{K}(b,b),D_{t})$ to $\{e_i\otimes 1\}_{1\leq i\leq n}$ as a basis of $H^{n}(\mathcal{C},D_{t})\otimes_{\mathcal{O}}\mathfrak{M}_t$. Then by Lemma \ref{lem3}, we have that
\begin{align}\label{eq527}
\partial^{(1)}(t^se_{\kappa}^{\underline{i}})
&=\partial(t^se_{\kappa}^{\underline{i}})
+\kappa^{\underline{i}}t^s\Big((\kappa-|\underline{i}|)e^{\underline{i}}\sum_{\mathbf{j}\in \mathcal{F}}\sum_{k\le \rho(\mathbf{j})+1}p^{\rho(\mathbf{j})}b(k,0;\mathbf{j})e_k\nonumber\\
&+\sum_{l=1}^ni_le_1^{i_1}\cdots e_l^{i_l-1}\cdots e_n^{i_n}\sum_{\mathbf{j}\in \mathcal{F}}\sum_{m\le \rho(\mathbf{j})+l+1}p^{\rho(\mathbf{j})}b(m,l;\mathbf{j})e_m\Big)\nonumber\\
&=\partial(t^se_{\kappa}^{\underline{i}})+ \sum_{\mathbf{j}\in \mathcal{F}}p^{\rho(\mathbf{j})}\mathbf{w}(\mathbf{j}),
\end{align}
where 
\begin{align*}
\mathbf{w}(\mathbf{j})&= \kappa^{\underline{i}}t^s\Big((\kappa-|\underline{i}|)e^{\underline{i}}\sum_{m\le \rho(\mathbf{j})+1}b(m,0;\mathbf{j})e_m
\\  &+\sum_{l=1}^ni_le_1^{i_1}\cdots e_l^{i_l-1}\cdots e_n^{i_n}\sum_{m\le \min\{\rho(\mathbf{j})+l+1,n\}}b(m,l;\mathbf{j})e_m\Big) \\
&=\sum_{m\leq\rho(\mathbf{j})+1}t^{s}b(m,0;\mathbf{j})(e_{1}^{i_{1}}\cdots e_{m}^{i_{m}+1}\cdots e_{n}^{i_{n}})_{\kappa} \\
&+\sum_{m\leq\min\{\rho(\mathbf{j})+l+1,n\}}i_{l}t^{s}b(m,l;\mathbf{j})(e_{1}^{i_{1}}\cdots e_{m}^{i_{m}+1}\cdots e_{l}^{i_{l}-1}\cdots e_{n}^{i_{n}})_{\kappa}
\end{align*}
such that $b(m,l;\mathbf{j})$ is of the form $\sum_{r(n+1)\le \rho(\mathbf{j})+l-m+1} a(r)\gamma^{(n+1)r} t^r$ with ${\rm ord}_p a(r)\ge 0$.
By Lemma \ref{lem3} we see that the total weight of each monomial term of the form $t^{r}e^{\underline{j}}_{\kappa}$ in $\mathbf{w}(\mathbf{j})$ is less than $(n+1)s+w(\underline{i})+\rho(\mathbf{j})+1$. 
Note that ${\rm ord}_pb(s,\underline{i})\ge b(n+1)s+cw(\underline{i})+c$. 
We find that $\sum b(s,\underline{i})p^{\rho(\mathbf{j})}\mathbf{w}(\mathbf{j})\in \mathcal{S}(b,c;0)$ such that ${\rm ord}_p(\sum b(s,\underline{i})p^{\rho(\mathbf{j})}\mathbf{w}(\mathbf{j}))\ge \rho(\mathbf{j})$.
Thus there exists  $\mathbf{w}(\mathbf{j}_1)\in \mathcal{S}(b,c;0)$ with ${\rm ord}_p\mathbf{w}(\mathbf{j}_1)\ge \rho(\mathbf{j}_1)$ such that
\begin{equation}\label{eq10}
\zeta=\sum_{r\ge 0,\underline{i}=(0,i_2,\cdots,i_n)}a(r;\underline{i})t^re_{\kappa}^{\underline{i}}+\partial(\eta_0)+\sum_{\mathbf{j}_1\in \mathcal{F}}\mathbf{w}(\mathbf{j}_1).
\end{equation}
Then from Remark \ref{remark} we conclude that for each $\mathbf{w}(\mathbf{j}_1)$, there exists $a^{\mathbf{j}_1}(r;\underline{i})\in \Omega$ with ${\rm ord}_p a^{\mathbf{j}_1}(r;\underline{i})\ge \rho(\mathbf{j}_1)$,
$\mathbf{w}(\mathbf{j}_1,\mathbf{j}_2)\in \mathcal{S}(b,c;0)$ with ${\rm ord}_p\mathbf{w}(\mathbf{j}_1,\mathbf{j}_2)\ge \rho(\mathbf{j}_1)+\rho(\mathbf{j}_2)$ and $\eta_1\in \mathcal{S}(b,c;c)$ with ${\rm ord}_p(\eta_1)\ge \rho(\mathbf{j}_1)$ such that
$$\mathbf{w}(\mathbf{j}_1)=\sum_{r\ge 0,\underline{i}=(0,i_2,\cdots,i_n)} a^{\mathbf{j}_1}(r;\underline{i})t^re_{\kappa}^{\underline{i}}+\sum_{\mathbf{j}_2\in \mathcal{F}}w(\mathbf{j}_1,\mathbf{j}_2)+ \partial(\eta_1).$$
Then
\begin{align*}
\zeta=&\sum_{\substack{r\ge 0\\ \underline{i}=(0,i_2,\cdots,i_n)}}a(r;\underline{i})t^re_{\kappa}^{\underline{i}}+ \sum_{\mathbf{j}_1\in \mathcal{F}}\Big(\sum_{\substack{r\ge 0\\ \underline{i}=(0,i_2,\cdots,i_n)\\
}}  a^{\mathbf{j}_1}(r;\underline{i})t^re_{\kappa}^{\underline{i}}
+\sum_{\mathbf{j}_2\in \mathcal{F}}w(\mathbf{j}_1,\mathbf{j}_2)+ \partial(\eta_1)\Big).
\end{align*}

Recursively doing above procedure we have that
\begin{align*}
\zeta=& \sum_{s\ge 1}\sum_{\mathbf{j}_1,\ldots,\mathbf{j}_s\in \mathcal{F}} \sum_{\substack{r\ge 0\\ \underline{i}=(0,i_2,\cdots,i_n)}} a^{\mathbf{j}_1,\cdots,\mathbf{j}_s}(r;\underline{i})t^re_{\kappa}^{\underline{i}} \\
+&\sum_{\substack{r\ge 0\\ \underline{i}=(0,i_2,\cdots,i_n)}}a(r;\underline{i})t^re_{\kappa}^{\underline{i}}
+\partial(\eta_0+\eta_1+\cdots+\eta_s),
\end{align*}
where ${\rm ord}_p \eta_j\ge \rho(\mathbf{j}_1)+\cdots+\rho(\mathbf{j}_{s})$ for $1\le j\le s$ and
$${\rm ord}_p  a^{\mathbf{j}_1,\cdots,\mathbf{j}_s}(r;\underline{i})\ge \rho(\mathbf{j}_1)+\cdots+\rho(\mathbf{j}_{s}).$$
 Theorem \ref{thm3} follows by letting  $s\rightarrow+\infty$.
\end{proof}

To give the explicit description of $H_{\kappa}^1(\mathcal{S}(b,c))$, it remains to show $\mathcal{R}(b,c)\cap \partial\mathcal{S}(b,c)=\{0\}.$
\begin{thm}\label{thm11}
 For $\kappa \in \mathbb{Z}_p\setminus \mathbb{Z}_{\ge 0}$, we have that
$$\mathcal{R}(b,c)\cap \partial\mathcal{S}(b,c)=\{0\}.$$ Then $H_{\kappa}^1(\mathcal{S}(b,c))\cong \mathcal{R}(b,c)$ and $H_{\kappa}^0(\mathcal{S}(b,c))=0$.
\end{thm}
\begin{proof}
First we show that $\mathcal{R}(b,c)\cap \partial\mathcal{S}(b,c)=\{0\}.$
Let $\eta\in \mathcal{R}(b,c)\cap \partial\mathcal{S}(b,c)$ and  $\zeta_0\in \mathcal{S}(b,c)$ such that
$\partial(\zeta_0)=\eta$. If $\zeta_0=0$, then we have done. We consider the case $\zeta_0\neq 0$.
 Then there exists $\epsilon\in \mathbb{R}$ such that $\zeta_0\in \mathcal{S}(b,c;\epsilon)$ but $\zeta_0\notin \mathcal{S}(b,c;c+\epsilon)$.
From (\ref{eq527}), we conclude that
\begin{align*}
\partial(\zeta_0)=\partial^{(1)}(\zeta_0)+\xi_1
\end{align*}
with $\xi_1\in \mathcal{S}(b,c;\epsilon)$.
Then by Theorem \ref{thm3}, there exists $\eta_1\in \mathcal{R}(b,c;\epsilon)$ and $\zeta_1\in \mathcal{S}(b,c;\epsilon+c)$ such that
$$\xi_1=\eta_1+\partial(\zeta_1).$$
Applying (\ref{eq527}) to $\zeta_1$, we have that
\begin{align*}
\partial(\zeta_1)=\partial^{(1)}(\zeta_1)+\xi_2
\end{align*}
with $\xi_2\in \mathcal{S}(b,c;\epsilon+c)$.
 Repeat the process, then we get
$$\eta=\partial(\zeta_0)=\partial^{(1)}(\sum _{i\ge 0}\zeta_i)+ \sum _{j\ge 1}\eta_j$$
with $\zeta_i \in \mathcal{S}(b,c;\epsilon+ci)$ and $\eta_j \in \mathcal{R}(b,c;\epsilon+c(j-1))$.
One has that $\sum _{i\ge 0}\zeta_i\in \mathcal{S}(b,c;\epsilon)$ and $\sum _{j\ge 1}\eta_j\in \mathcal{R}(b,c;\epsilon)$ are well defined since the valuations in $\zeta_i$ and $\eta_j$ are increased by $c$.
By Lemma \ref{lem5}, we have
$\partial^{(1)}\mathcal{S}(b,c) \cap \mathcal{R}(b,c)=\{0\}.$
It follows that
$$\zeta_0=-\sum _{i\ge 1}\zeta_i$$
which contradicts the choice of $\epsilon$.
Hence $\zeta_0=0$. It follows that $H_{\kappa}^1(\mathcal{S}(b,c))\cong \mathcal{R}(b,c)$.

If we set $\eta=0$, one has ${\rm ker} \partial=0$. It follows that $H_{\kappa}^0(\mathcal{S}(b,c))=0$.

\end{proof}

\section{$p$-adic estimations}
From now on, we let $p>2$ and $\frac{1}{p-1}<b\leq 1$ and $c=b-\frac{1}{p-1}$. 
Let $q=p^a$ with $a\ge 0$. Recall that  $\bar{\alpha}_a$ is the Frobenius map from $ H^n(\mathcal{K}(b,b),D_{t})$ to $H^n(\mathcal{K}_q(b,b),D_{t^q})$. 
It follows  from Theorem \ref{thm2.1} and the definition of $\bar{\alpha}_a$ that $\bar{\alpha}_a(e_0)\equiv e_0 \mod (\gamma,t)$. 
Then we can linearly extend $\bar{\alpha}_a$ to  $\mathcal{S}(b,c)$ as follows. Define $[\bar{\alpha}_a]_{\infty,\kappa}:\mathcal{S}(b,c)\rightarrow \Omega[[t,e^{\underline{i}}]]$ by
$$[\bar{\alpha}_a]_{\infty,\kappa}(e_{\kappa}^{\underline{i}}):=\kappa^{\underline{i}}\bar{\alpha}_a(e_0)^{\kappa-|\underline{i}|}\bar{\alpha}_a(e_1)^{i_1}\cdots \bar{\alpha}_a(e_n)^{i_n}.$$
In what follows, we will show $[\bar{\alpha}_a]_{\infty,\kappa}$ is well defined.

Define $\mathcal{S}_{k,q}:={\rm Sym}^{k}H^n(\Omega^{\bullet}(\mathcal{K}_{q}(b,b),D_{t^q}))$ to be a $p$-adic space
over $\mathcal{L}_{q}(b)$ with basis
$\{e_{0}^{i_{0}}e^{\underline{i}}:=e_0^{i_0}e_1^{i_1}\cdots e_n^{i_n}\}$ such that $k=i_0+i_1+\cdots+i_n$.
Let $I_k$ be the set $\{(i_0,\underline{i})\in \mathbb{Z}^{n+1}:i_0+|\underline{i}|=k\}$.
Define the Frobenius map $Sym^k \bar{\alpha}_a:\mathcal{S}_{k,1}\rightarrow \mathcal{S}_{k,q}$ by
$$Sym^k \bar{\alpha}_a(e_{0}^{i_{0}}e^{\underline{i}}):=\bar{\alpha}_a(e_0)^{i_0}\bar{\alpha}_a(e_1)^{i_1}\cdots\bar{\alpha}_a(e_n)^{i_n}$$
with $k=i_0+i_1+\cdots+i_n$.
In particular, if we write
$$Sym^k \bar{\alpha}_1(e_{0}^{i_{0}}e^{\underline{i}})=\sum_{(j_0,\underline{j})\in I_k}\widetilde{A}^{(k)}((j_{0},\underline{j}),(i_{0},\underline{i}))e_{0}^{j_{0}}e^{\underline{j}} \mod D_{t^p},$$
then \cite[Theorem 4.1]{HS172} gives the $p$-adic estimate of $\widetilde{A}^{(k)}((j_{0},\underline{j}),(i_{0},\underline{i}))$.
\begin{lem}\label{lem3.2} For every $(i_0,\underline{i}),(j_0,\underline{j})\in I_k$, $\widetilde{A}^{(k)}((j_{0},\underline{j}),(i_{0},\underline{i}))\in \mathcal{L}(\frac{1}{p-1};w(\underline{j}))$.
\end{lem}

Now let $a=0$ and write
$$[\bar{\alpha}_1]_{\infty,\kappa}(e^{\underline{i}}_{\kappa})=\sum_{\underline{j}\in\mathbb{Z}_{\ge0}^n} A^{\infty,\kappa}(\underline{j},\underline{i})e^{\underline{j}}_{\kappa}.$$
\begin{lem}\label{lem3.3}
 For $\kappa \in \mathbb{Z}_p\setminus \mathbb{Z}_{\ge 0}$, we have that $[\bar{\alpha}_1]_{\infty,\kappa}:\mathcal{S}(b,c)\rightarrow \mathcal{S}(b/p,c)$ is well-defined.
\end{lem}
\begin{proof}
From Theorem \ref{thm2.1}, we may write
$$\bar{\alpha}_{1}(e_0)=1+\sum_{j=0}^{n}\widetilde{A}(j,0)e_{j},$$
where $\widetilde{A}(0,0)=A(0,0)-1$ and $\widetilde{A}(j,0)=A(j,0)\in\mathcal{L}(\frac{1}{p-1};j)$ for $j\ge 1$.
Then ${\rm ord}_p\widetilde{A}(0,0)\ge\frac{1}{p-1}$.
We have
\begin{align*}
\bar{\alpha}(1)^{\kappa}&=\sum_{l=0}^{\infty}\binom{\kappa}{l}(\sum_{j=0}^{n}\widetilde{A}(j,0)e_{j})^{l} \\
&=\sum_{l=0}^{\infty}\binom{\kappa}{l}\sum_{\begin{subarray}{c} l_0+l_{1}+\cdots+l_{n}=l \\ l_0,l_{1},\cdots,l_{n}\geq 0
\end{subarray}}\binom{l}{l_{0},l_{1},\cdots,l_{n}}\prod_{j=0}^{n}\widetilde{A}(j,0)^{l_{j}}e_{j}^{l_{j}}.
\end{align*}
Then we have that
\begin{align}\label{eq41}
[\bar{\alpha}_1]_{\infty,\kappa}(e^{\underline{i}}_{\kappa})&=\kappa^{\underline{i}}\bar{\alpha}(1)^{\kappa-|\underline{i}|}\prod_{s=1}^{n}\bar{\alpha}_1(e_{s})^{i_{s}}\nonumber\\
=&\kappa^{\underline{i}}\sum_{l=0}^{\infty}\binom{\kappa-|\underline{i}|}{l}\sum_{\begin{subarray}{c} l_{0}+\cdots+l_{n}=l \\ l_{0},\cdots,l_{n}\geq 0
\end{subarray}}\binom{l}{l_{0},\cdots,l_{n}}\prod_{j=0}^{n}\widetilde{A}(j,0)^{l_{j}}e_{j}^{l_{j}}\sum_{(j_0,\underline{j})\in I_{|\underline{i}|}}\widetilde{A}^{(|\underline{i}|)}(\underline{j},\underline{i})e^{\underline{j}}\nonumber\\
=&\kappa^{\underline{i}}\sum_{l=0}^{\infty}\sum_{\begin{subarray}{c} l_{0}+\cdots+l_{n}=l \\ l_{0},\cdots,l_{n}\geq 0
\end{subarray}}\sum_{(j_{0},\underline{j})\in I_{|\underline{i}|}} 
\Big(\binom{\kappa-|\underline{i}|}{l}\binom{l}{l_{0},\cdots,l_{n}}
\widetilde{A}^{(|\underline{i}|)}((j_{0},\underline{j}),(i_{0},\underline{i}))\nonumber\\
&\prod_{j=0}^{n}\widetilde{A}(j,0)^{l_{j}}e^{\underline{j}+(l_{1},\cdot\cdot\cdot,l_{n})}\Big)\nonumber\\
=&\sum_{\underline{\mu}\in\mathbb{Z}_{\geq 0}^{n}}A^{\infty,\kappa}(\underline{\mu},\underline{i})e^{\underline{\mu}}_{\kappa},
\end{align}
where
\begin{equation*}
\begin{aligned}
&A^{\infty,\kappa}(\underline{\mu},\underline{i})\\
&=\sum_{l=0}^{\infty}\sum_{\begin{subarray}{c} l_{0}+\cdots+l_{n}=l \\ l_{0},\cdots,l_{n}\geq 0
\end{subarray}}\sum_{\begin{subarray}{c}
 \underline{\mu}=\underline{j}+(l_{1},\cdot\cdot\cdot,l_{n})
\end{subarray}}\frac{\kappa^{\underline{i}}}{\kappa^{\underline{\mu}}}
\binom{\kappa-|\underline{i}|}{l}\binom{l}{l_{0},\cdots,l_{n}}\widetilde{A}^{(|\underline{i}|)}((j_{0},\underline{j}),(i_{0},\underline{i}))\prod_{j=0}^{n}\widetilde{A}(j,0)^{l_{j}}.    
\end{aligned}    
\end{equation*}
 Note that $|\underline{\mu}|=|\underline{j}|+l_{1}+\cdots+l_{n}=l+|\underline{i}|-j_{0}-l_{0}$.  We have $$\frac{\kappa^{\underline{i}}}{\kappa^{\underline{\mu}}}\binom{\kappa-|\underline{i}|}{l}\binom{l}{l_{0},\cdots,l_{n}}=\frac{j_{0}!}{(l-l_{0})!}\binom{\kappa-l-|\underline{i}|+j_{0}+l_{0}}{j_{0}+l_{0}}\binom{j_{0}+l_{0}}{l_{0}}\binom{l-l_{0}}{l_{1},\cdots,l_{n}}.$$ Kummer's theorem shows that this combinatorial coefficient lies in $\frac{1}{(l-l_{0})!}\mathbb{Z}_{p}$. Hence
 $$\ord_{p}\frac{\kappa^{\underline{i}}}{\kappa^{\underline{\mu}}}\binom{\kappa-|\underline{i}|}{l}\binom{l}{l_{0},\cdots,l_{n}}\geq -\frac{l-l_{0}}{p-1}.$$
 Applying the estimates of Lemma \ref{lem3.2} that $\widetilde{A}^{(|\underline{i}|)}((j_{0},\underline{j}),(i_{0},\underline{i}))\in \mathcal{L}(1/(p-1);w(\underline{j}))$ and $\widetilde{A}(j,0)^{l_{j}}\in\mathcal{L}(1/(p-1);jl_j) $.
Also note that $\omega(\underline{\mu})=\omega(\underline{j})+\sum_{j=1}^{n}jl_{j}$. Then we easily obtain
\begin{equation}\label{eq42}
A^{\infty,\kappa}(\underline{\mu},\underline{i})\in\mathcal{L}(\frac{1}{p-1};\omega(\underline{\mu})-\frac{l}{p-1}+\frac{l_0}{p-1}).
\end{equation}
Note that $c=b-\frac{1}{p-1}\le 1-\frac{1}{p-1}$. Hence
$\omega(\underline{\mu})-\frac{l}{p-1}+\frac{l_0}{p-1}\ge (1-\frac{1}{p-1})\omega(\mu)$.
It follows that
$\sum A^{\infty,\kappa}(\underline{\mu},\underline{i})e^{\underline{\mu}} \in \mathcal{S}(b/p,c)$.
\end{proof}

Define map $\psi_{t}:\mathcal{S}(b/p,c)\to\mathcal{S}(b,c)$ by
$$\psi_{t}: \sum a(r;\underline{i})t^re^{\underline{i}}_{\kappa}\mapsto \sum a(pr;\underline{i})t^re^{\underline{i}}_{\kappa}.$$
Then we define $\beta_{\infty,\kappa,1}:\mathcal{S}(b,c)\rightarrow\mathcal{S}(b,c)$ by $\beta_{\infty,\kappa,1}:=\psi_{t}\circ [\bar{\alpha}_1]_{\infty,\kappa}$ and
$\beta_{\infty,\kappa}:=\psi_{t}^a\circ [\bar{\alpha}_a]_{\infty,\kappa}.$
It has been proved in \cite[Section 3.2]{YZ23}
that $\beta_{\infty,\kappa}$ is completely continuous and
$$L( Sym^{\kappa,\infty}{\rm Kl}_n,T)=\det(1-\beta_{\infty,\kappa}T|\mathcal{S}_{\infty,\kappa})^{\delta_q},$$
where $g(T)^{\delta_q}=g(T)/g(qT).$

By an analogous argument as \cite[Lemma 3.7]{H17}, we have that $q\partial\circ \beta_{\infty,\kappa}=\beta_{\infty,\kappa}\circ \partial$. Then $\beta_{\infty,\kappa}$ induces a map $\bar{\beta}_{\infty,\kappa}$
on $H^1(\mathcal{S}(b,c),\partial)$
 and $H^0(\mathcal{S}(b,c),\partial)$.
Then we have
$$L( Sym^{\kappa,\infty}{\rm Kl}_n,T)=\frac{\det(1-\bar{\beta}_{\infty,\kappa}T|H^{1}(\mathcal{S}(b,c),\partial))}
{\det(1-\bar{\beta}_{\infty,\kappa}qT|H^{0}(\mathcal{S}(b,c),\partial))}.$$
It follows from Theorem \ref{thm11} that $H^{0}(\mathcal{S}(b,c),\partial)=0$.
Hence
$$L( Sym^{\kappa,\infty}{\rm Kl}_n,T)=\det(1-\bar{\beta}_{\infty,\kappa}T|H^{1}(\mathcal{S}(b,c),\partial)).$$

 By the same argument as \cite[Theorem 4.9]{H17} we have $\bar{\beta}_{\infty,\kappa,1}^a=\bar{\beta}_{\infty,\kappa}$. We first give the $p$-adic
 estimation for $\bar{\beta}_{\infty,\kappa,1}$.
For each $(r,\underline{i}),(s,\underline{j})\in \mathcal{B}_1$, where $$\mathcal{B}_{1}=\{(r,\underline{i})|r\in\mathbb{Z}_{\geq 0}, \underline{i}=(0,i_{2},i_{3},\cdots)\in\mathbb{Z}_{\geq 0}^{n}\},$$
$\{\gamma^{(n+1)r}t^{r}e^{\underline{i}}_{\kappa}\}_{(r,\underline{i})\in\mathcal{B}_{1}}$ is a basis of $H^{1}(S(b,c),\partial)$. For this basis, there exists $B^{\infty,\kappa}((s,\underline{j});(r,\underline{i}))\in\Omega$ such that
$$\bar{\beta}_{\infty,\kappa,1}(\gamma^{(n+1)r}t^re^{\underline{i}}_{\kappa})=\sum_{(s,\underline{j})\in \mathcal{B}_1} B^{\infty,\kappa}((s,\underline{j});(r,\underline{i}))\gamma^{(n+1)s}t^se^{\underline{j}}_{\kappa}.$$

Let $b=1$ and $c=1-1/(p-1)$. Then we have the following estimation.

\begin{thm}\label{thm63}
For $(r,\underline{i}),(s,\underline{j})\in \mathcal{B}_1$, we have ${\rm ord}_{p}B^{\infty,\kappa}((s,\underline{j});(r,\underline{i}))\geq (1-\frac{1}{p-1})W(s;\underline{j})$.
\end{thm}
\begin{proof}
By the definition and (\ref{eq41}), we write 
\begin{align*}
\beta_{\infty,\kappa,1}(\gamma^{(n+1)r}t^{r}e^{\underline{i}}_{\kappa})
&=\psi_{t}(\sum_{\underline{i}'\in\mathbb{Z}_{\geq 0}^{n}}\gamma^{(n+1)r}t^{r}A^{\infty,\kappa}(\underline{i}',\underline{i})e^{\underline{i}'}_{\kappa}).
\end{align*}
Using (\ref{eq42}), we may write $A^{\infty,\kappa}(\underline{i}',\underline{i})=\sum_{l\geq 0}A^{\infty,\kappa}(\underline{i}',\underline{i};l)t^{l}$ with
$${\rm ord}_{p}A^{\infty,\kappa}(\underline{i}',\underline{i};l)\geq\frac{(n+1)l}{p-1}+(1-\frac{1}{p-1})\omega(\underline{i}').$$
Then
\begin{align*}
\beta_{\infty,\kappa,1}(\gamma^{(n+1)r}t^{r}e^{\underline{i}}_{\kappa})&=\psi_{t}(\sum_{\underline{i}'\in\mathbb{Z}_{\geq 0}^{n},\ l\geq 0}\gamma^{(n+1)r}A^{\infty,\kappa}(\underline{i}',\underline{i};l)t^{l+r}e^{\underline{i}'}_{\kappa}) \\
&=\sum_{\underline{i}'\in\mathbb{Z}_{\geq 0}^{n},\ l\geq r}\gamma^{(n+1)(r-l)}A^{\infty,\kappa}(\underline{i}',\underline{i};pl-r)\gamma^{(n+1)l}t^{l}e^{\underline{i}'}_{\kappa}
\end{align*}
Then one has that $\beta_{\infty,\kappa,1}(\gamma^{(n+1)r}t^{r}e^{\underline{i}}_{\kappa})\in \mathcal{S}(1,1-\frac{1}{p-1};0)$.
Then by Theorem \ref{thm11} we have that
$$\bar{\beta}_{\infty,\kappa,1}(\gamma^{(n+1)r}t^{r}e^{\underline{i}}_{\kappa})=\sum_{(s,\underline{j})\in\mathcal{B}_{1}} B^{(r,\underline{i})}(s,\underline{j}) t^s e^{\underline{j}}_{\kappa}\in \mathcal{R}(1,1-\frac{1}{p-1};0).$$
Hence $B^{\infty,\kappa}((s,\underline{j});(r,\underline{i}))=\sum_{(s,\underline{j})\in\mathcal{B}_{1}} B^{(r,\underline{i})}(s,\underline{j})\gamma^{-(n+1)s}$ with
$${\rm ord}_p(B^{(r,\underline{i})}(s,\underline{j})\gamma^{-(n+1)s} )\ge (n+1)s+(1-\frac{1}{p-1})w(\underline{j})-\frac{(n+1)s}{p-1}=(1-\frac{1}{p-1})W(s;\underline{j}).$$

\end{proof}

\begin{thm}\label{thm632}
 If we write 
 $$R(T)/(1-T^{n+1})=\sum_{i=0}^{\infty}h_i(n)T^i,$$
then the $q$-adic Newton polygon of $L(\kappa,T)$ lies on or above the $q$-adic Newton polygon of $$\prod(1-q^{(1-\frac{1}{p-1})i}T)^{h_i(n)}.$$ \end{thm}
 \begin{proof}
Note that $L(\kappa,T)$ is continuous with respect to $\kappa$.
It is enough to let $\kappa\in \mathbb{Z}_p\setminus \mathbb{Z}_{\ge 0}$.
By the above arguments, we have that 
$$L( Sym^{\kappa,\infty}{\rm Kl}_n,T)=\det(1-\bar{\beta}_{\infty,\kappa}T|H^{1}(\mathcal{S}(b,c),\partial))$$
and 
$$\bar{\beta}_{\infty,\kappa,1}^a=\bar{\beta}_{\infty,\kappa}.$$

On the other hand, we have
\begin{align}\label{eq44}
\det(1-\bar{\beta}_{\infty,\kappa}T^a|H^{1}(\mathcal{S}(b,c),\partial))&=\det (1-\bar{\beta}_{\infty,\kappa,1}^aT^a|H^{1}(\mathcal{S}(b,c),\partial))\\\nonumber
&=\prod_{\zeta^a=1} \det (1-\zeta\bar{\beta}_{\infty,\kappa,1}T|H^{1}(\mathcal{S}(b,c),\partial)).
\end{align}
Let $m_i$ denote the number of reciprocal roots of $\det (1-\bar{\beta}_{\infty,\kappa,1}T|H^{1}(\mathcal{S}(b,c),\partial))$ which has slope $s_i$. Then by (\ref{eq44}) we have that $\det(1-\bar{\beta}_{\infty,\kappa}T|H^{1}(\mathcal{S}(b,c),\partial))$ has $m_i$ reciprocal roots of $q$-adic slope $s_i$.
Hence Theorem \ref{thm632} follows from Theorem \ref{thm63}  and Dwork's argument \cite[Section 7]{Dw64}.
 \end{proof}

\section{The $k$-th symmetric power $L$-function}

As an application, we use the $q$-adic Newton polygon of $L( Sym^{k,\infty}{\rm Kl}_n,T)$ to establish a uniform lower bound for the $k$-th symmetric power $L$-function $L(Sym^k {\rm Kl}_n,T)$ on those segments with slopes $\le k$.

Define 
\begin{align*}
L(k,\bar{t}):=&\prod_{\bar{t}\in |\mathbb{G}_m/\mathbb{F}_q|}\prod_{i_1+\cdots+i_n>k}\\&\frac{1}{1-\pi_0(\bar{t})^{k-i_1-\cdots-i_n-\frac{2(k+1)}{n(n+1)}}\pi_1(\bar{t})^{i_1-\frac{2(k+1)}{n(n+1)}}\cdots\pi_n(\bar{t})^{i_n-\frac{2(k+1)}{n(n+1)}}(q^{k+1}T)^{\deg(\bar{t})}}.
\end{align*}
It follows from the functional equation $\pi_0(\bar{t})\pi_1(\bar{t})\cdots\pi_n(\bar{t})=q^{\frac{n(n+1)}{2}\deg(\bar{t})}$ that 
$$L( Sym^k{\rm Kl}_n,T)=\frac{L( Sym^{k,\infty}{\rm Kl}_n,T)}{L(k,\bar{t})}.$$
Note that  $i_1-\frac{2(k+1)}{n(n+1)}+\cdots+n(i_n-\frac{2(k+1)}{n(n+1)})\ge 0$.
Then the denominator $L(k,\bar{t})$ has all slope $\ge k+1$.
Then the Newton polygon of $L( Sym^k{\rm Kl}_n,T)$ coincides with the Newton polygon of $L( Sym^{k,\infty}{\rm Kl}_n,T)$ on the segments with slopes $\le k$. 



\begin{rmk}
In their study of the one-dimensional Kloosterman family, Fres\'{a}n-Sabbah-Yu's \cite{FSY} improved the $q$-adic lower bound of the $k$-th symmetric power $L$-function by eliminating the coefficient $1-1/(p-1)$ in the Hodge polygon. By taking the limit, it follows that the estimate for the infinite symmetric power $L$-function obtained by Haessig \cite{H17} can also be improved.
We conjecture that the same coefficient $1-1/(p-1)$ in the Hodge polygon also can be removed for the infinite symmetric power $L$-function of the hyper-Kloosterman family.
\end{rmk}
\begin{center}
{\sc Acknowledgements}
\end{center}
L.P. Yang is supported by National Natural Science Foundation of China (Grant No. 12201078).
We are grateful to Prof. Wan Daqing for many enlightening discussions. We are also grateful to Zhang Hao and Zhang Dingxin for their valuable suggestions.

\bibliographystyle{amsplain}

\begin{thebibliography}{10}
\bibitem{AS2}A. Adolphson and S. Sperber, Hyperkloosterman sums revisited, {\it J. Number Theory} {\bf243}(2023), 328-352.
\bibitem{De}P. Deligne, La conjecture de Weil. I, {\it 
Inst. Hautes \'{E}tudes Sci. Publ. Math.} {\bf43} (1974), 273-307.
\bibitem{Dw64} B. Dwork, On the zeta function of a hypersurface.II, {\it Ann. of Math.} {\bf 80} (1964),227-299.
\bibitem{Dw73} B. Dwork, Normalized period matrices II, {\it Ann. of Math.} {\bf 98} (1973),1-57.

\bibitem{FSY}J. Fres\'{a}n, C. Sabbah, and J.-D. Yu, Hodge theory of Kloosterman connections, {\it Duke Math. J.} {\bf171} (2022), 1649-1747.


\bibitem{FW1} L. Fu and D. Wan, $L$-functions for symmetric products of Kloosterman sums, {\it J. Reine Angew. Math.}    {\bf589} (2005),
79-103.


\bibitem{FW3} L. Fu and D. Wan, Trivial factors for $L$-functions of symmetric products of Kloosterman sheaves, {\it Finite Fields Appl.} {\bf 14} (2008), 549-570.




 \bibitem{H17} C. D. Haessig,  $L$-functions of symmetric powers of Kloosterman sums (unit root $L$-functions and $p$-adic estimates), {\it Math. Ann.} {\bf 369} (2017), 17-47.

 \bibitem{HS172} C. D. Haessig and S. Sperber, Symmetric power $L$-functions for families of generalized Kloosterman sums, {\it Trans. Amer. Math. Soc.} {\bf 369} (2017), 1459-1493.

   \bibitem{HS24} C. D. Haessig and S. Sperber, Symmetric Power $L$-functions of the hyper-Kloosterman Family,arXiv:2402.13051.
   
   \bibitem{Qin24}Y. Qin, Hodge numbers of motives attached to Kloosterman and Airy moments, {\it J. Reine Angew. Math.} {\bf 808} (2024), 143-192.
   
\bibitem{Qin24-2}Y. Qin, $L$-functions of Kloosterman sheaves, {\it Proc. Lond. Math. Soc.} {\bf 129}(5) (2024).


     \bibitem{Ro86} P. Robba, Symmetric powers of the $p$-adic Bessel equation, {\it J. Reine Angew. Math.} {\bf 366} (1986), 194-220.
 \bibitem{SP0} S. Sperber, $p$-adic hypergeometric functions and their cohomology, {\it Duke Math. J.} {\bf 44} (1977), 535-589.

\bibitem{SP1} S. Sperber, Congruence properties of the hyperkloosterman sum, {\it Comp. Math.} {\bf 40} (1980), 3-33.

\bibitem{Wan99} D. Wan, Dwork’s conjecture on unit root zeta functions, {\it Ann. of Math. } {\bf 150}
(1999), 867-927.
\bibitem{YZ23}L.P. Yang and H. Zhang, Unit roots of the unit root $L$-functions of Kloosterman
family, {\it Finite Fields Appl.} {\bf 92} (2023), Paper No. 102293.


\end{thebibliography}

\end{document}